\documentclass[12pt]{amsart}
\usepackage[margin=0.8in]{geometry}

\usepackage{amssymb,amsmath,amsthm,amsfonts,framed}
\usepackage{framed,comment,enumerate}
\usepackage{xcolor, framed}
\usepackage{tikz}
\usepackage{color, framed}
\usepackage{graphicx,caption,subcaption}

\usepackage[colorlinks]{hyperref}
\hypersetup{linkcolor=blue,citecolor=blue,filecolor=black,urlcolor=blue}

\usepackage{verbatim} 

\usepackage[showonlyrefs]{mathtools}

\usepackage[sort,nocompress]{cite}

\def\R {\mathbb{R}}
\def\C {\mathcal{C}}
\def\N {\mathbb{N}}
\def\S {\mathbb{S}}

\def\eps{\varepsilon}
\def\dist{{\rm dist}}

\newcommand{\loc}{\mathrm{loc}}

\newcommand{\pa}{\partial}

\newtheorem{proposition}{Proposition}[section]
\newtheorem{theorem}[proposition]{Theorem}
\newtheorem*{theorem*}{Theorem}
\newtheorem{corollary}[proposition]{Corollary}
\newtheorem{lemma}[proposition]{Lemma}

\theoremstyle{definition}
\newtheorem{definition}[proposition]{Definition}
\newtheorem{remark}[proposition]{Remark}
\numberwithin{equation}{section}

\renewcommand{\le}{\leqslant}

\renewcommand{\ge}{\geqslant}

\begin{document}

\title[Fractional elliptic equations in Lipschitz epigraphs]{On fractional elliptic equations \\
in Lipschitz sets and epigraphs: \\ regularity, monotonicity and rigidity results}

\author[S. Dipierro]{Serena Dipierro}
\address[Serena Dipierro]{School of Mathematics and Statistics,
University of Melbourne, 813 Swanston St, Parkville VIC 3010, Australia,
and
School of Mathematics and Statistics,
University of Western Australia,
35 Stirling Highway
Crawley, Perth
WA 6009,
Australia}
\email{sdipierro@unimelb.edu.au}

\author[N. Soave]{Nicola Soave}
\address[Nicola Soave]{Mathematisches Institut,
Justus-Liebig-Universit\"at Giessen, 
Arndtstrasse 2, 35392 Giessen, Germany}
\email{nicola.soave@gmail.com, nicola.soave@math.uni-giessen.de}

\author[E. Valdinoci]{Enrico Valdinoci}
\address[Enrico Valdinoci]{School of Mathematics and Statistics,
University of Melbourne, 813 Swanston St, Parkville VIC 3010, Australia,
Weierstra{\ss}-Institut f\"ur Angewandte
Analysis und Stochastik, Mohrenstra{\ss}e 39, 10117 Berlin, Germany, 
Dipartimento di Matematica, Universit\`a degli studi di Milano,
Via Saldini 50, 20133 Milan, Italy, and
School of Mathematics and Statistics,
University of Western Australia,
35 Stirling Highway
Crawley, Perth
WA 6009,
Australia.}
\email{enrico@mat.uniroma3.it}

\begin{abstract}
We consider a nonlocal equation set in an unbounded domain
with the epigraph property. We prove symmetry, monotonicity and rigidity results.
In particular, we deal with halfspaces, coercive epigraphs and epigraphs that
are flat at infinity. 

These results can be seen as the nonlocal counterpart of the celebrated
article~\cite{BCNCPAM}.
\end{abstract}

\maketitle

\section{introduction}

The study of monotonicity and rigidity of solutions to semilinear elliptic equations of fractional order in the whole space $\R^N$ 
or in smooth bounded sets $\Omega$ has attracted considerable attention in the last years, 
see e.g. \cite{BaScMo, CaCi, CaCi2, CaSi, CaSi2, ChenLiLi, DDW, DipMonPerSci, Faz, FeWa, SiVa}. 
In striking contrast, if $\Omega$ is unbounded, but different from the whole space, 
very few results are available, all concerning the particular case of the half-space $\Omega=\R^N_+$, see \cite{FaJa, QuaXia}, 
or the one of exterior sets, see \cite{LiLi, JarNod, SoVa}. 
The main purpose of this paper is the study of the qualitative properties of bounded solutions to 
\begin{equation}\label{problem}
\begin{cases}
(-\Delta)^s u = f(u) & \text{in $\Omega$}, \\
u>0 & \text{in $\Omega$}, \\
u=0 & \text{in $\R^N \setminus \Omega$},
\end{cases} 
\end{equation}
where~$\Omega$ is assumed to be
\emph{the epigraph of a continuous function $\varphi:\R^{N-1} \to \R$}, i.e.
we suppose that
\begin{equation}\label{OMEGA}
\Omega:= \left\{ x \in \R^N: x_N > \varphi(x')\right\}\,, \qquad {\mbox{ with }}x' = (x_1,\dots,  x_{N-1}) \in \R^{N-1}.
\end{equation}
Notice that the half-space $\R^N_+$ falls within this definition with $\varphi \equiv 0$. 

In \eqref{problem}, $(-\Delta)^s$ with $s \in (0,1)$ denotes the fractional Laplacian, which can be defined as the operator acting on sufficiently smooth functions as
\[
\begin{split} 
(-\Delta)^s u(x) & := c_{N,s}\,PV\int_{\R^N} \frac{u(x)-u(y)}{|x-y|^{N+2s}}\,dy \\
& = c_{N,s} \lim_{\eps \to 0^+} \int_{\R^N \setminus B_\eps(x)} \frac{u(x)-u(y)}{|x-y|^{N+2s}}\,dy,
\end{split}
\]
where~$c_{N,s}>0$ is a normalizing constant (which plays no major role in the present
paper and which will be often omitted for simplicity), and $PV$ stays for ``principal value".

We will consider different assumptions on $\varphi$, obtaining different monotonicity 
and rigidity properties accordingly. 

The nonlinearity~$f$ in~\eqref{problem} belongs to a reasonably wide class of functions, including for instance those of bistable-type (a precise definition
will follow shortly). Under these assumptions,
the main results of this paper are:
\begin{itemize}
\item boundary regularity, monotonicity and further qualitative properties for 
solutions to~\eqref{problem} in globally Lipschitz epigraphs;
\item monotonicity for 
solutions to~\eqref{problem} in
coercive epigraphs;
\item $1$-dimensional symmetry in the half-space;
\item rigidity for overdetermined problems 
in epigraphs that are sufficiently ``flat at infinity". 
\end{itemize}
Similar results in the classical case~$s=1$ were obtained in the
seminal paper~\cite{BCNCPAM}. Here, dealing with a nonlocal framework, a careful analysis is needed to overcome
the lack of explicit barriers and several ad-hoc arguments will be exploited
to replace the study of the point-wise behavior of the solution with
a global study of the geometry of the problem.\medskip

As 
additional statements, we also derive a very general maximum principle 
(tailor-made for non-decaying solutions in possibly unbounded domains), a general version
of the sliding method for the fractional Laplacian, and a boundary regularity result 
for solutions of fractional boundary value problems in sets satisfying an 
exterior cone condition. \medskip

Before proceeding with the statement of our results, we clarify that with the terminology
\emph{solution} in this paper we always mean \emph{classical solution}.

As a matter of fact, without extra effort, the
same results would apply to bounded
\emph{viscosity solutions} of~\eqref{problem}: indeed, since we will assume that
the nonlinearity~$f$ is locally Lipschitz continuous, the regularity theory for viscosity solutions
(developed in \cite{CafSilCPAM, CafSilAnn}) implies that viscosity and classical solutions coincide in our setting
(see~\cite[Remark 2.3]{QuaXia} for a detailed explanation). 

In addition, we mention that~\emph{distributional (i.e. very weak)
solutions} or \emph{weak solutions} (as defined e.g. in \cite{FaJa, SoVa}) could be
considered as well with minor changes. 

For the reader's convenience, we will recall the definition of classical and viscosity solution at the end of the introduction.

In the next subsections, we describe in details the results obtained. In all the forthcoming statements, the fractional parameter $s$ will always be a fixed value in the interval $(0,1)$.

\subsection{Boundary value problems in globally Lipschitz epigraphs}

In this subsection we consider the case in which
the domain~$\Omega$ 
of~\eqref{problem} is a \emph{globally Lipschitz epigraph}. Namely,
we suppose that the function $\varphi$ in~\eqref{OMEGA}
is globally Lipschitz continuous, with Lipschitz constant $K$.

On the nonlinearity $f$, we suppose that:
\begin{itemize}
\item[($f1$)] $f$ is locally Lipschitz continuous in $\R$, and there exists $\mu>0$ such that 
$f(t) >0$ for any~$t\in (0,\mu)$, and $f(t) \le 0$ for any~$t \ge \mu$;
\item[($f2$)] there exist ${t_{0}} \in (0,\mu)$, and $\delta_0>0$ such that $f(
t) \ge \delta_0 \,t$ for any $t\in[0,{t_{0}}]$;
\item[($f3$)] there exists ${t_{1}} \in ({t_{0}},\mu)$ such that $f$ is non-increasing in 
$({t_{1}},\mu)$.
\end{itemize} 
As prototype example, we may think at $f(t) =t-t^3$, which yields the fractional 
Allen-Cahn equation, that is a widely studied model in phase transitions
in media with long-range particle interactions, see e.g.~\cite{SavValGamma}. \medskip

The first of our main results is the natural counterpart of Theorems 1.1 and 1.2 in \cite{BCNCPAM}. 

\begin{theorem}\label{thm: main 1}
Let $\Omega$ be a globally Lipschitz epigraph. Let $f$ satisfy assumptions ($f1$)-($f3$), and let $u$ be a bounded solution to \eqref{problem}. Then:
\begin{itemize}
\item[($i$)] $u <\mu$ in $\Omega$.
\item[($ii$)] As $\dist(x,\pa \Omega) \to +\infty$, we have that~$u(x) \to \mu$ uniformly in $\Omega$.
\item[($iii$)] There exist $C,\bar \rho,h_1>0$ such that 
\[
u(x',x_N) \ge C(x_N-\varphi(x'))^{\bar \rho} \qquad  \text{if $x_N-\varphi(x') < h_1$}.
\]
\item[($iv$)] $u$ is globally $\alpha$-H\"older continuous in $\R^N$, for some $\alpha \in (0,s)$.
\item[($v$)] $u$ is the unique bounded solution to \eqref{problem}.
\item[($vi$)] If $(a_1,\dots,a_{N-1})$ is such that $$\sum_i a_i^2 < K^{-2},$$ then  
\[
\pa_{x_N} u + \sum_i a_i \pa_{x_i} u >0 \qquad \text{in $\Omega$}.
\]
In particular, $u$ is monotone increasing in $x_N$.
\end{itemize}
\end{theorem}

We stress that, since $\Omega$ is merely a Lipschitz set and the exterior sphere condition is not necessarily satisfied along $\pa \Omega$, the H\"older continuity of the solution does not follow by previous contributions (see Subsection \ref{au45678IUHAUaJ13} for more details). Both the exponents $\bar \rho$ and $\alpha$ appearing in the theorem are determined by the choice of $\varphi$. To be more precise, we note that a globally Lipschitz epigraph satisfy both a uniform exterior cone condition with angle $\theta_1$ and a uniform interior cone condition with angle $\theta_2$ (see Subsection \ref{au45678IUHAUaJ13} for a precise definition of the exterior cone condition); then, as it will be clear from the proofs, the exponent $\alpha$ depends on $\theta_1$, while the index $\bar \rho$ depends on $\theta_2$.

Regarding point ($vi$) in the theorem, it establishes that $u$ is monotone increasing in any direction $\tau$ such that there exists an orthonormal basis $y_1,\dots,y_N$ of $\R^N$ with $y_N= \tau$, and in the new coordinates $\Omega$ is still the epigraph of a globally Lipschitz function $\psi= \psi(y_1,\dots,y_{N-1})$. Thus, in the particular case $\varphi \equiv 0$, i.e. when $\Omega=\R^N_+$ is a half-space, we deduce monotonicity and $1$-dimensional symmetry of the solutions.

\begin{corollary}\label{cor: half-space}
Let $\Omega=\R^N_+$, and let $f$ satisfy ($f1$)-($f3$), and let $u$ be a
bounded solution to \eqref{problem}. Then $u$ depends only on $x_N$, and
\[
\pa_{x_N} u>0 \qquad \text{in $\Omega$}.
\]
\end{corollary}

Previous results regarding monotonicity of solutions to nonlocal equations in 
half-spaces 
can be found in \cite{FaWe, QuaXia} (see also~\cite{SavValMono}
for results in the whole of~$\R^N$), where the authors dealt with non-decreasing nonlinearities~$f$ satisfying $f(0) \ge 0$. 

We also refer to \cite{BaDPGMQu} (which appeared after the present paper was submitted), where the authors proved that any bounded, non-negative and non-trivial solution to \eqref{problem} with $f$ of class $\C^1$ is montone increasing. 

In all the articles~\cite{BaDPGMQu,FaWe,QuaXia} the monotonicity is used to derive non-existence results for non-decreasing nonlinearities~$f$, which is a complementary situation 
with respect to the one considered here. 

The $1$-dimensional symmetry in the half-space,
as addressed in Corollary~\ref{cor: half-space}, was, 
up to now, open. 

Coming back to the monotonicity of the solutions, we emphasize that the main result in \cite{BaDPGMQu} allows to treat also the case $f(0)<0$. This marks a relevant difference with the local setting~$s=1$, since in case $f(0)<0$ non-negative solutions of local equations are not necessarily monotone, and only partial results are available (we refer the interested reader to 
\cite{BCNpisa, Cort, Da, FaSc, FaSo}).

%

The proof of Theorem \ref{thm: main 1} is given in Section \ref{sec: thm 1},
and relies on some classical ideas of~\cite{BCNCPAM} -- nevertheless,
all the intermediate steps present several substantial difficulties 
of purely nonlocal nature. As a matter of fact,
in \cite{BCNCPAM} the authors often construct more or less 
explicit local barriers, and exploit local properties of functions whose Laplacian 
has a strict sign. On the other hand, the construction of a barrier 
function is much harder when dealing with integro-differential 
operators, since such barrier has to be defined in the all space $\R^N$, and 
has to satisfy a boundary condition on the complement of a certain set $D$ (and 
not only on $\partial D$). Moreover, local properties of functions 
cannot be inferred by the only knowledge of the fractional Laplacian 
in some neighbourhood and any modification of the function ``far away''
affects the values
of its fractional Laplacian at a point. These are just two sources of new obstructions 
which we shall overcome; we refer to the comments and the remarks 
written throughout the paper for further details.

\subsection{Monotonicity of solutions in coercive epigraphs}

In this subsection, we deal with
the case in which the domain~$\Omega$ of~\eqref{problem}
is a \emph{coercive epigraph}, namely, we suppose that the function~$\varphi: \R^{N-1} \to \R$
in~\eqref{OMEGA}
is continuous and satisfies
\[
\lim_{|x|' \to +\infty} \varphi(x') = +\infty.
\]
In this setting, we have the following result:
\begin{theorem}\label{thm: main 2}
Let $\Omega$ be a coercive epigraph, and let $u$ be a solution (not necessarily bounded) to\begin{equation*}
\begin{cases}
(-\Delta)^s u = f(x,u) & \text{in $\Omega$}, \\
u > 0 & \text{in $\Omega$}, \\
u=0 & \text{in $\R^N \setminus \Omega$},
\end{cases}
\end{equation*}
with $f(x,t)$ continuous in $\overline{\Omega} \times \R$, non-decreasing in $x_N$, 
and locally Lipschitz continuous in $t$, locally uniformly in $x$, in the following sense: for any $M>0$ and 
any compact set~$K \subset \overline{\Omega}$, there exists $C>0$ such that
\[
\sup_{x \in K} \frac{|f(x,t)-f(x,\tau)|}{|t-\tau|} \le C \qquad {\mbox{ for any }} t, \tau \in [-M,M].
\]
Then $u$ is monotone increasing in $x_N$.
\end{theorem}

This result is the natural counterpart of~\cite[Theorem 1.3]{BCNCPAM}, 
which in turn is a refinement of \cite[Proposition II.1]{EstLio}. 
Its proof rests on the moving planes method for the fractional Laplacian.

\subsection{Overdetermined problems for the fractional Laplacian in epigraphs.} 

In this subsection, we consider the overdetermined setting in
which both Dirichlet and Neumann conditions are prescribed in problem~\eqref{problem}.
Differently from the classical case, the Dirichlet condition needs to be set
in the complement of the domain (and not along its boundary)
and the Neumann assumptions takes into account (in a suitable sense)
normal derivatives of fractional order.

For this,
given an open set $\Omega$ with~$\mathcal{C}^2$ boundary,
we denote by~$\nu(x_0)$ the inner unit normal vector
at~$x_0\in\partial\Omega$. For any~$u \in \C^{0,s}(\R^N)$ and~$x_0 \in \pa \Omega$,
we consider the \emph{outer normal $s$-derivative} of $u$ in $x_0$, defined as
\begin{equation}\label{def: s-derivative}
(\pa_{\nu})_s u(x_0) := -\lim_{t \to 0^+} \frac{u(x_0 + t \nu(x_0))-u(x_0)}{t^s}.
\end{equation}
The boundary regularity 
theory for fractional Laplacian, developed 
in~\cite{GrubbAdv, GrubbAnal, RosSerJMPA, RosSerDuke}, ensures that, for a solution $u$ to \eqref{problem} with $\Omega$ of class $\mathcal{C}^2$, the quantity $(\pa_\nu)_s u$ is well defined. Natural Hopf's Lemmas were then proved 
in~\cite[Proposition 3.3]{FaJa} and~\cite[Lemma 1.2]{GreSer}, and constituted the base point in the study of overdetermined problems for the fractional Laplacian, 
see~\cite{Dalibard, FaJa, GreSer, SoVa, LiLi}.

In this paper we consider overdetermined problems of the type
\begin{equation}\label{overdet}
\begin{cases}
(-\Delta)^s u = f(u) & \text{in $\Omega$}, \\
u> 0 & \text{in $\Omega$}, \\
u=0 & \text{in $\R^N \setminus \Omega$}, \\
(\pa_{\nu})_s u = const. & \text{on $\pa \Omega$}.
\end{cases}
\end{equation}
We will suppose that~$\Omega$ is the epigraph of a $\mathcal{C}^{2}$ and 
globally Lipschitz 
function $\varphi: \R^{N-1} \to \R$, satisfying the following additional assumption:
\begin{equation}\label{hp epigrafico}
\text{for any $\tau \in \R^{N-1}$, uniformly in $x'$, } \lim_{|x'| \to +\infty} (\varphi(x'+\tau) - \varphi(x')) = 0.
\end{equation}
This condition, firstly proposed in \cite{BCNCPAM}, can be seen
as a flatness condition of~$\pa \Omega$ at infinity. 

We can extend \cite[Theorem 7.1]{BCNCPAM} in the nonlocal setting.

\begin{theorem}\label{thm: overdet}
Let $\Omega$ be the epigraph of a $\mathcal{C}^2$ and globally Lipschitz function $\varphi: \R^{N-1} \to \R$, 
satisfying \eqref{hp epigrafico}. Let $f$ satisfy ($f1$)-($f3$), and let us suppose that \eqref{overdet} has a bounded solution $u$. 
Then $\Omega$ is a half-space $\{x_N > const.\}$, and $u$ depends only on $x_N$ and is monotone increasing in $x_N$.
\end{theorem}

Theorem \ref{thm: overdet} is proved in Section \ref{sec: overdet}.

\subsection{Boundary regularity in domains satisfying an
exterior cone condition}\label{au45678IUHAUaJ13}

In this subsection, we obtain general boundary regularity results for solutions to 
\begin{equation}\label{p_r}
\begin{cases}
(-\Delta)^s u = g(x) & \text{in $\Omega$}, \\
u = 0 & \text{in $\R^N \setminus \Omega$}.
\end{cases}
\end{equation}
In this setting, we will not restrict to the case in which~$\Omega$
is an epigraph, but we will assume instead that~$\Omega$ satisfies an 
exterior cone condition with some uniform opening $\theta \in (0,\pi)$.
More precisely, 
for a given direction~$e\in \S^{N-1}$ and a given angle~$\theta\in(0,\pi)$,
we denote by~$\Sigma_{e,\theta}$ the open, rotationally symmetric cone 
of axis~$e$ and opening~$\theta$ (that is the set of all vectors~$v\in\R^n$
that form with~$e$ an angle less than~$\theta$).
We suppose that
there exists $\theta\in(0,\pi)$ such that: if $x \in \partial \Omega$, then 
for a direction $e \in \S^{N-1}$ the cone $x + \Sigma_{e,\theta}$ is exterior
to $\Omega$ and tangent to $\Omega$ in $x$. 

We point out that in this definition the opening $\theta$ is 
the same for all the points of $\pa \Omega$, while the direction $e$ can change.
We also observe that globally Lipschitz epigraphs $\{x_N > \varphi(x')\}$ enjoy 
the uniform exterior cone condition, with $\theta$ depending only on the
Lipschitz constant of $\varphi$. 

The boundary regularity of solutions to boundary value problems driven by integro-differential 
operators has been object of several contributions \cite{Bog, GrubbAdv, GrubbAnal, RosSerpre, RosSerJMPA, RosSerDuke}. 
As far as we know, for boundary value problems of type \eqref{p_r} the minimal assumption on $\Omega$ was considered in \cite[Proposition 1.1]{RosSerJMPA}, 
where the authors supposed that $\Omega$ is a bounded Lipschitz domain satisfying a uniform exterior ball condition. 
The following theorem weaken this assumption, establishing the global H\"older continuity of bounded solutions to \eqref{p_r} when $\Omega$ 
satisfies a uniform exterior cone condition. 

\begin{theorem}\label{thm: boundary regularity}
Let $\Omega$ be a possibly unbounded open set of $\R^N$, satisfying the 
uniform exterior cone condition with opening $\theta \in (0,\pi/2)$. Let $u \in L^\infty(\R^N)$ be a solution to \eqref{p_r}, 
with $g \in L^\infty(\Omega)$. 

Then, there exist $\alpha \in (0,s)$ and $C>0$, both depending only on $\theta$, $s$ and $N$, such that $u \in \mathcal{C}^{0,\alpha}(\R^N)$, and
\begin{equation}\label{ST:T}
\|u\|_{\mathcal{C}^{0,\alpha}(\R^N)} \le  C \left[
\left(1+ \|u\|_{L^\infty(\R^N)}\right) \left(1+\|g\|_{L^\infty(\Omega)}^{\frac{\alpha}{2s-2\alpha}}\right) + \|g\|_{L^\infty(\Omega)} \right].
\end{equation}
Also, for any $s \in (0,1)$, the map $\theta \mapsto \alpha(\theta,s)$ is monotone non-decreasing.
\end{theorem}

\begin{remark}
We stress that it is not even necessary to suppose the Lipschitz regularity of $\Omega$.

If $\theta \ge \pi/2$, then the uniform exterior cone condition yields a uniform exterior sphere condition, 
and hence by \cite[Proposition 1.1]{RosSerJMPA} solutions to fractional boundary value problems in $\Omega$ 
are already known to be of class $\mathcal{C}^{0,s}(\R^N)$. This is why we only consider $\theta \in (0,\pi/2)$ in 
Theorem~\ref{thm: boundary regularity}.
\end{remark}

\begin{remark}
The proof of Theorem \ref{thm: boundary regularity} is based upon the construction of a wall of barriers, whose definition uses essentially the existence of homogeneous solutions to non-local problems in cones, contained in \cite{BanBog}. The homogeneity exponent appearing in this construction coincides precisely with the regularity index $\alpha$, and in particular the monotonicity of $\alpha$ in the angle $\theta$ follows from Lemma 3.3 in \cite{BanBog} (see Section \ref{sec: boundary}, and in particular Remark \ref{rem: on monotonicity in theta}, for more details).  
\end{remark}

For our purposes, the importance of \eqref{ST:T}
is to provide uniform convergence of sequences of solutions 
under very reasonable assumptions.
Namely, let us consider a sequence $\{u_n\}$ of solutions to
\[
\begin{cases}
(-\Delta)^s u_n = g_n & \text{in $\Omega_n$}, \\
u_n = 0 & \text{in $\R^N \setminus \Omega_n$}, 
\end{cases}
\]
with $\{u_n\}$ and $\{g_n\}$ uniformly bounded in $L^\infty(\R^N)$ 
and $L^\infty(\Omega_n)$, respectively. Then Theorem \ref{thm: boundary regularity} 
implies that $u_n \to u_\infty$ locally uniformly in $\R^N$, up to a subsequence. \medskip

The proof of Theorem \ref{thm: boundary regularity} is the object of Section \ref{sec: boundary}.

\subsection{Some useful results of independent interest}

We conclude the introduction stating some general results 
that are auxiliary to the proof of the main theorems and
which we think are also of independent interest.

\medskip
 
In proving Theorem \ref{thm: main 1}, a crucial tool will be a maximum principle in 
unbounded domain for the fractional Laplacian. This is the fractional counterpart of \cite[Theorem 2.1]{BCNCPAM}, 
and we stress that while in the local case the domain $D$ is supposed to be connected, this is not necessary in the nonlocal setting.

\begin{theorem}\label{thm: super max}
Let $D$ be an open set in $\R^N$, possibly unbounded and disconnected. 
Suppose that $\overline{D}$ is disjoint from the closure of an infinite open connected
cone. Let $z \in \mathcal{C}(\R^N)$ bounded above, and satisfying in viscosity sense
\begin{equation}\label{EQ:z}
\begin{cases}
(-\Delta)^s z - c(x) z \le 0 & \text{in $D$}, \\
z \le 0 & \text{in $\R^N \setminus D$},
\end{cases}
\end{equation}
for some $c \in L^\infty(D)$, $c \le 0$ a.e. in $D$, with $c z \in \C(D)$. Then $z \le 0$ in $D$.
\end{theorem}

We observe that Theorem~\ref{thm: super max} here improves~\cite[Theorem 2.4]{MR3427047},
where a similar result was proved when~$D$ was a half-space.

We shall also need the following version of the sliding method for the fractional Laplacian. 

\begin{theorem}\label{thm: sliding}
Let $\Omega$ be a bounded open subset of $\R^N$, convex in the direction $e_N=(0',1)$. Let $w$ be a solution of
\begin{equation*}
\begin{cases}
(-\Delta)^s w = g(x,w) & \text{in $\Omega$}, \\
w=\varphi & \text{in $\R^N \setminus \Omega$},
\end{cases}
\end{equation*}
with $g(x,t)$ continuous in $\overline{\Omega} \times \R$, non-decreasing in $x_N$, 
and locally Lipschitz continuous in $t$, uniformly in $x$, in the following sense: for any $M>0$, there exists $C>0$ such that
\[
\sup_{x \in \Omega} \frac{|g(x,t)-g(x,\tau)|}{|t-\tau|} \le C \qquad {\mbox{ for any }} t,\tau \in [-M,M].
\]
On the boundary term $\varphi$, we suppose that for every $x=(x',x_N)$, $y=(y',y_N)$, $z=(z',z_N)$ with $x_N < y_N < z_N$, it holds
\begin{equation}\label{hp phi sliding}
\begin{split}
\varphi(x) < w(y) < \varphi(z) \qquad & \text{if $y \in \Omega$}, \\
\varphi(x) \le \varphi(y) \le  \varphi(z) \qquad & \text{if $y \in \R^N \setminus \Omega$}. 
\end{split}
\end{equation}
Then $w$ is monotone increasing with respect to $x_N$.
\end{theorem}
Thanks to the maximum principle in sets of small measure \cite[Proposition 2.2]{FeWa}, the proof of Theorem \ref{thm: sliding} is a straightforward adaptation of the local one, given in \cite{BeNi}, and hence is omitted.

\medskip

\noindent{\textbf{Basic definitions and notations:}} we start recalling the definition of classical solution.

\begin{definition}
A continuous function $u : \R^N \to \R$ is a \emph{classical solution} to
\begin{equation}\label{generic equation}
(-\Delta)^s u = h \quad \text{in $\tilde \Omega$}, \quad u = g \quad \text{in $\R^N \setminus \tilde \Omega$}
\end{equation}
 if $(-\Delta)^s u(x)$ is well defined and equal to $h(x)$ for every $x \in \tilde\Omega$, and $u = g$ a.e. in $\R^N \setminus \tilde\Omega$.   
\end{definition}

Notice that, beyond the continuity, no boundary regularity is required on $u$. It is well known that a sufficient condition to ensure that $(-\Delta)^s u$ is well defined (and actually continuous) in an open set $U \subset \R^N$ is that $u \in \C^{2s +\eps}(U)$ for some $\eps>0$ (i.e. $u \in \C^{0,2s+\eps}(U)$ if $s < 1/2$, or $u \in  \C^{1,2s+\eps-1}(U)$ if $s \ge 1/2$), see \cite[Proposition 2.4]{SilCPAM}.

For future convenience, we recall here also the definition of viscosity solution (see \cite[Definition 2.2]{CafSilCPAM}.

\begin{definition}
A function $u: \R^N \to \R$, upper (resp. lower) semicontinuous is said to be a \emph{viscosity sub-solution} (resp. \emph{viscosity super-solution}) to \eqref{generic equation}, if $u \le g$ (resp $u \ge g$) a.e. in $\R^N \setminus \tilde \Omega$, and if every time all the following happen:
\begin{itemize}
\item $x_0 \in \tilde \Omega$;
\item $N$ is a neighbourhood of $x_0 \in \tilde \Omega$;
\item $\phi$ is some $\C^2$ function in $N$; 
\item $\phi(x_0) = u(x_0)$; 
\item $u(x) < \phi(x)$ (resp. $u(x) > \phi(x)$) for every $x \in N \setminus \{x_0\}$;  
\end{itemize}
then if we let 
\[
v:= \begin{cases} \phi & \text{in $N$} \\ u & \text{in $\R^N \setminus N$}, \end{cases} 
\]
we have $(-\Delta)^s v (x_0) \le h(x_0)$ (resp. $(-\Delta)^s v(x_0) \ge h(x_0)$). A function is a viscosity solution if it both a viscosity sub- and super-solution.
\end{definition}

We adopt in the rest of the paper a mainly standard notation. The ball of center $x$ and radius $r$ is denoted by $B_r(x)$, and in the frequent case $x=0$ we simply write $B_r$. The letter $C$ always denotes a positive constant, whose precise value is allowed to change from line to line.  

\medskip

\noindent \textbf{Organization of the paper:} 
Section~\ref{8iujn9ijn7yhbnikjm} deals with
the maximum principle in unbounded domains, providing the proof
of Theorem \ref{thm: super max}.

Then, Section~\ref{sec: boundary}
is devoted to the
boundary regularity in sets satisfying an exterior cone condition, and contains the
proof of Theorem \ref{thm: boundary regularity}.

The monotonicity in globally Lipschitz epigraphs
and the proof of Theorem \ref{thm: main 1} are dealt with
in Section~\ref{sec: thm 1}, while
the monotonicity of solutions in coercive epigraphs, with the 
proof of Theorem \ref{thm: main 2}, is the subject of Section~\ref{8uhg6tr4rtgbifdfghjAHJ}.

Finally, in Section~\ref{sec: overdet}, we consider
overdetermined problems and we prove
Theorem~\ref{thm: overdet}.

\section{Maximum principle in unbounded domains}\label{8iujn9ijn7yhbnikjm}

We devote this section to the proof of Theorem \ref{thm: super max}.
To this aim, we start with some preliminary results. First of all, we notice that
balls centered at boundaries
of cones intersect the cones with mass proportional to that of the ball, namely:

\begin{lemma}\label{LEMMA CONE 1}
Let~$\alpha\in\left(0,\frac\pi4\right]$ and~${\mathcal{C}}=\{ x=(x',x_N)\in\R^N
{\mbox{ s.t. }} x_N\ge|x|\cos\alpha\}$. Let~$p\in\partial{\mathcal{C}}$. Then,
for any~$r>0$,
\begin{equation} \label{Brp}
|B_r(p)\cap {\mathcal{C}}|\ge\delta r^N,\end{equation}
for some~$\delta>0$, depending on~$N$ and~$\alpha$.
\end{lemma}

\begin{proof} Up to scaling, we can assume that~$r=1$, so
we want to show that, for any~$p\in\partial{\mathcal{C}}$,
\begin{equation} \label{Brp:XOCC}
|B_1(p)\cap {\mathcal{C}}|\ge\delta ,\end{equation}
for some~$\delta>0$. For a contradiction, we suppose
that~\eqref{Brp:XOCC} is false, namely there exists a sequence~$p_k\in
\partial{\mathcal{C}}$ for which
\begin{equation} \label{Brp:XOCC:2} 
|B_1(p_k)\cap {\mathcal{C}}|\le\frac1k.\end{equation}
One can see that~$|p_k|$ needs to be bounded
(indeed, if~$p\in \partial{\mathcal{C}}$
and~$|p|$ is large enough, then
the ball of radius~$1/4$ and tangent from the inside to~$p$
lies in~$B_1(p)\cap {\mathcal{C}}$).
Therefore,
up to a subsequence, we have that~$p_k\to\bar p$,
for some $\bar p\in\partial{\mathcal{C}}$.
Hence, by using the Dominated Convergence Theorem,
we can pass~\eqref{Brp:XOCC:2} to the limit as~$k\to+\infty$
and obtain that
$$ |B_1(\bar p)\cap {\mathcal{C}}|=0,$$
which is a contradiction with the Lipschitz regularity of the cone.

For a different proof, see~\cite{OLD}.
\end{proof}

Here is another auxiliary results concerning the geometry of cones:

\begin{lemma}\label{LEMMA CONE 2}
Let~$\alpha\in\left(0,\frac\pi4\right]$ and~${\mathcal{C}}$ as in Lemma~\ref{LEMMA CONE 1}.

Let~$q\in\R^N\setminus {\mathcal{C}}$. Let~$r:=2\,{\rm dist}\,(q, {\mathcal{C}})$.
Then,
\begin{equation*} 
|B_r(q)\cap {\mathcal{C}}|\ge\delta r^N,\end{equation*}
for some~$\delta>0$, depending on~$N$ and~$\alpha$.
\end{lemma}

\begin{proof} Let~$p\in \partial {\mathcal{C}}$ be such that
$$ |p-q|= {\rm dist}\,(q, {\mathcal{C}})=\frac{r}{2}.$$
We claim that
\begin{equation}\label{JAK:AKaA}
B_{r/2}(p)\subseteq B_r(q).
\end{equation}
Indeed, if~$x\in B_{r/2}(p)$ then
$$ |x-q|\le |x-p|+|p-q|< \frac{r}{2}+\frac{r}{2}=r.$$
Then, \eqref{JAK:AKaA} and Lemma~\ref{LEMMA CONE 1}
imply the desired result.
\end{proof}

We are now in position to complete the proof of Theorem \ref{thm: super max}.

\begin{proof}[Completion of the proof of Theorem \ref{thm: super max}]
We consider~$z^+:=\max\{z,0\}$. We claim that
\begin{equation}\label{VIS:GL}
(-\Delta)^s z^+ -c(x)z^+\le0 \qquad \text{in $\R^N$},
\end{equation}
in the viscosity sense.

To prove this, let~$\phi$ be a smooth function touching~$z^+$ from
above at some point~$x_0$. We have two cases: either~$z(x_0)>0$
or~$z(x_0)\le0$.

Suppose first that~$z(x_0)>0$. Then~$x_0\in D$ and~$z(x_0)=z^+(x_0)$.
Accordingly,
$$ \phi(x)\ge z^+(x)\ge z(x) \quad{\mbox{ and }}\quad
\phi(x_0)=z^+(x_0)=z(x_0),$$
that is,
$\phi$ touches~$z$ from
above at~$x_0\in D$. Thus, by~\eqref{EQ:z}, we have that
$(-\Delta)^s \phi(x_0) - c(x_0) \phi(x_0) \le 0 $.

Now we consider the case in which~$z(x_0)\le0$.
Then
$$ \phi(x)\ge z^+(x)\ge 0 \quad{\mbox{ and }}\quad
\phi(x_0)=z^+(x_0)=0.$$
As a consequence, $\phi$ has a minimum at~$x_0$ and therefore~$(-\Delta)^s
\phi(x_0)\le0$. Accordingly, we have that
$$ (-\Delta)^s \phi(x_0) - c(x_0) \phi(x_0) \le 0+0=0.$$
This completes the proof of~\eqref{VIS:GL}.

{F}rom~\eqref{VIS:GL} and the fact that~$c\le0$ we conclude that
\begin{equation}\label{VIS:GL:2}
(-\Delta)^s z^+\le0 \qquad \text{in~$\R^N$},
\end{equation}
in the viscosity sense.

Now, in order to complete the proof of
Theorem \ref{thm: super max}, we want to show that~$z^+$ vanishes identically.
Suppose not: then there exists~$q_o\in \R^N$ such that~$z^+(q_o)>0$.
Then, recalling that $z$ is bounded from above, we set
$$ A:=\sup_{\R^N} z^+ \ge z^+(q_o)>0,$$
and we take a maximizing sequence~$q_j\in\R^N$ such that
\begin{equation}\label{JA9} \lim_{j\to+\infty} z^+(q_j)=A.\end{equation}
Of course, up to neglecting a finite number of indices, we may suppose that~$z^+(q_j)\ge A/2>0$.
Hence, since, by~\eqref{EQ:z}, we know that~$z\le0$ outside $D$, we have that~$q_j\in D$.
We denote by~${\mathcal{C}}$ a cone that lies outside~$D$
(whose existence is warranted by assumption). Then we have that~$q_j\in \R^N\setminus
{\mathcal{C}}$, and we set
$$ r_j:=2\,{\rm dist}\,(q_j,{\mathcal{C}})>0.$$
Notice also that~${\mathcal{C}}\subseteq \R^N\setminus D\subseteq\{z^+=0\}$.
So,
by Lemma~\ref{LEMMA CONE 2}, we know that
$$ \delta r_j^N \le |B_{r_j}(q_j)\cap {\mathcal{C}}|\le
|B_{r_j}(q_j)\cap\{z^+=0\}|,$$
for some~$\delta>0$.

Thus,
we are\footnote{In order to apply Corollary 4.5 in \cite{MR2244602}, we observe that assumptions (4.7), (4.8) and (4.10) therein have been already verified. As far as (4.9) is concerned, we recall that as exponent $\eta>0$ associated to $(-\Delta)^s$ we have to take a positive small number (see \cite[Section 2]{MR2244602}). Then, for any $x \in \R^N \setminus B_{r_j}(q_j)$, it results that
\[
A \left( 2 \left|2 \frac{x-q_j}{r_j}\right|^\eta-1 \right) \ge A \left(2^{1+\eta} -1 \right) \ge A \ge z^+(x),
\]
i.e. assumption (4.9) holds.}
in the position of
applying~\cite[Corollary 4.5]{MR2244602}
for $z^+$ in each ball $B_{r_j}(q_j)$, and we obtain that~$z^+\le(1-\gamma)A$ in~$B_{r_j/2}(q_j)$,
for some~$\gamma\in(0,1)$. But this says that
$$ z^+(q_j)\le (1-\gamma)A$$
and so, taking the limit and recalling~\eqref{JA9}, we obtain~$A\le(1-\gamma)A$,
which is a contradiction.
\end{proof}

We conclude this section recalling the strong maximum principle for the fractional Laplacian, whose simple proof is omitted for the sake of brevity.

\begin{proposition}[Strong maximum principle]\label{prop: strong}
Let~$\Omega \subset \R^N$ be an open set, neither necessarily unbounded, 
nor connected. Let $w$ be a classical solution to 
\[
\begin{cases}
(-\Delta)^s w + c(x) w \ge 0 & \text{in $\Omega$}, \\
w \ge 0 & \text{in $\R^N$}, 
\end{cases}
\]
with $c \in L^\infty(\Omega)$ and $c w \in \mathcal{C}(\Omega)$. Then either $w>0$, or $w \equiv 0$ in $\R^N$.
\end{proposition}

\section{Boundary regularity for the fractional Laplacian
in sets satisfying an exterior cone condition}\label{sec: boundary}

In this section we analyze the global regularity of solutions to the boundary 
value problem \eqref{p_r} and we prove
Theorem \ref{thm: boundary regularity}. Here we will assume that~$
g \in L^\infty(\Omega)$ and $\Omega$ is a set satisfying the uniform exterior 
cone condition with opening $\theta$, as defined in Subsection~\ref{au45678IUHAUaJ13}.

We recall that, for $e \in \S^{N-1}$ and $\theta \in [0,\pi]$, we denote by $\Sigma_{e,\theta}$ the open cone of rotation axis $\R e$ and opening $\theta$. In particular, if $e=e_N$, then 
\[
\Sigma_{e_N,\theta} = \begin{cases} \{x \in \R^N: |x'| <  (\tan{\theta}) x_N\} & \text{if }\theta \in \left[0,\frac{\pi}{2}\right), \\
\{x \in \R^N: x_N >0\} & \text{if } \theta=\frac{\pi}{2}, \\
\{x \in \R^N: |x'| > (\tan{\theta}) x_N\}  & \text{if }\theta \in \left(\frac{\pi}{2},\pi\right].
\end{cases} 
\]
In this framework,
Theorem \ref{thm: boundary regularity} follows as a corollary of the next intermediate statement: 

\begin{proposition}\label{prop: 4.1 in BCN} 
Let~$u\in L^\infty(\R^N)$ be a solution of~\eqref{p_r}.
Then, there exist~$\alpha \in (0,s)$ and $C>0$ depending only on $\theta$ such that 
\begin{equation}\label{7899oiHYGFAGHJhgfghj11OA}
|u(x)| \le C\,\big(1+\|u\|_{L^\infty(\R^N)}\big) \,\left( 1+  \|g\|_{L^\infty(\Omega
)}^{\frac{\alpha}{2s-2\alpha}}\right)\;  \dist(x,\pa \Omega)^\alpha  
\end{equation}
for all $x \in \Omega$. 
\end{proposition}

The proof of Proposition \ref{prop: 4.1 in BCN} is divided into several lemmas. We produce a wall of upper barriers for $u$, whose construction is inspired by \cite[Lemma 4.1]{BCNCPAM}. Major difficulties arise in our setting in order to compute the fractional Laplacian of the barriers. 

In the next lemma we establish an estimate for $\alpha$-homogeneous functions, which can be seen as a one-side counterpart of the classical Euler formula. 

\begin{lemma}\label{lem: frac Euler}
Let $\alpha>0$, $0<\theta_1 <\theta_2< \pi$, and let $v \in \mathcal{C}^1_{\loc}(\Sigma_{e_N,\theta_2})$ be a positive $\alpha$-homogeneous function in $\Sigma_{e_N; \theta_2}$, non-negative in the whole space $\R^N$. Then,
there exists $C>0$ depending on $\theta_1$, $\theta_2$ and $v$  such that
\begin{equation}\label{Euler}
\int_{\R^N} \frac{(v(x)-v(y))^2}{|x-y|^{N+2s}}\,dy \ge C |x|^{2\alpha-2s} \qquad \text{for every $x \in \Sigma_{e_N,\theta_1}$}.
\end{equation}
\end{lemma} 

\begin{proof}
For any $A\subseteq\R^N$, we define
$$ {\mathcal{F}}_A (x):=\int_A \frac{\big(v(x)-v(y)
\big)^2}{|x-y|^{N+2s}}\,dy.$$
For any $t>0$, the fact that $v$ is $\alpha$-homogeneous
and the substitution $z:= y/t$ imply that
\begin{eqnarray*}
&& {\mathcal{F}}_{\R^N} (tx_o) = 
\int_{\R^N} \frac{\big(t^\alpha v(x_o)-v(y)
\big)^2}{|tx_o-y|^{N+2s}}\,dy \\
&&\qquad\qquad= t^{2\alpha-2s}\,
\int_{\R^N} \frac{\big(v(x_o)-v(z)
\big)^2}{|x_o-z|^{N+2s}}\,dz = t^{2\alpha-2s}\,{\mathcal{F}}_{\R^N} (x_o)
.\end{eqnarray*}
In particular, by taking $x_o:=x/|x|$ and $t:=|x|$, we obtain that, for any $x\in
\Sigma_{e_N,\theta_1}$,
\begin{equation}\label{yuzy}
{\mathcal{F}}_{\R^N} (x) = |x|^{2\alpha-2s}\,{\mathcal{F}}_{\R^N}\left( \frac{x}{|x|}\right).
\end{equation}
Now we set ${\mathcal{A}}:= B_2\setminus B_{1/2}$ and
$$ \mu:=\inf_{x\in{{\mathcal{A}} \cap \Sigma_{e_N,\theta_1}}} 
{\mathcal{F}}_{ {\mathcal{A}} \cap \Sigma_{e_N,\theta_1} } (x).$$
We have that $\mu\ge0$. We claim that, in fact,
\begin{equation}\label{mu co}
\mu>0.
\end{equation}
Indeed, suppose by contradiction that $\mu=0$. Then, there exists a sequence $q_k$
in~${{\mathcal{A}} \cap \Sigma_{e_N,\theta_1}}$ such that
\begin{equation}\label{mu co2}
0=\mu= \lim_{k\to+\infty} {\mathcal{F}}_{ {\mathcal{A}} \cap \Sigma_{e_N,\theta_1} } (q_k)=
\lim_{k\to+\infty} 
\int_{ {\mathcal{A}} \cap \Sigma_{e_N,\theta_1} } \frac{\big(v(q_k)-v(y)
\big)^2}{|q_k-y|^{N+2s}}\,dy.\end{equation}
Also, since ${{{\mathcal{A}} \cap \Sigma_{e_N,\theta_1}}}$ is pre-compact,
up to a subsequence we may assume that $q_k\to q$ as $k\to+\infty$,
for some~$q$ in the closure of~${{{\mathcal{A}} 
\cap \Sigma_{e_N,\theta_1}}}$.
This, \eqref{mu co2} and Fatou Lemma imply that
$$ 0\ge \int_{ {\mathcal{A}} \cap \Sigma_{e_N,\theta_1} } 
\lim_{k\to+\infty}
\frac{\big(v(q_k)-v(y)
\big)^2}{|q_k-y|^{N+2s}}\,dy=
\int_{ {\mathcal{A}} \cap \Sigma_{e_N,\theta_1} } \frac{\big(v(q)-v(y)
\big)^2}{|q-y|^{N+2s}}\,dy.$$
Consequently, 
$v(y)=v(q)$ for any $y\in {\mathcal{A}} \cap \Sigma_{e_N,\theta_1}$,
that is $v$ is constant on ${\mathcal{A}} \cap \Sigma_{e_N,\theta_1}$
and thus 
\begin{equation}\label{na co}
{\mbox{$\nabla v=0$ on ${\mathcal{A}} \cap \Sigma_{e_N,\theta_1}$.}}\end{equation}
Now, since $v$ is $\alpha$-homogeneous,
$$ \nabla v(x)\cdot x=\alpha v(x)$$
and so, from the positivity assumption on $v$,
we deduce that, fixed $q_o\in \Sigma_{e_N,\theta_1}$ with $|q_o|=1$,
$$ \nabla v(q_o)\cdot q_o=\alpha v(q_o)>0.$$
This is in contradiction with \eqref{na co} and so it proves \eqref{mu co}.

Hence, using \eqref{yuzy} and \eqref{mu co},
for any $x\in\Sigma_{e_N,\theta_1}$,
$$
{\mathcal{F}}_{\R^N} (x) = |x|^{2\alpha-2s}\,{\mathcal{F}}_{\R^N}\left( \frac{x}{|x|}\right)
\ge \mu\,|x|^{2\alpha-2s}\,,$$
as desired.
\end{proof}

Now, for the sake of simplicity, we suppose that \begin{equation}\label{78909uyHJAHH678911}
0 \in \pa \Omega,\end{equation}
that~$\theta\in\left(0,\,\frac\pi2\right)$ and that the cone $\Sigma_{-e_N,\theta}$ is exterior to $\Omega$. Then we take $0<\bar \theta_2<\bar \theta_1:= \theta< \pi/2$, so that 
\[
\Sigma_{-e_N, \bar \theta_2} \subset   \Sigma_{-e_N, \bar \theta_1} \subset (\R^N \setminus \Omega).
\]
Notice that the admissible range of $\bar \theta_2$ depends only on the opening $\theta$. Setting $\theta_i:= \pi-\bar \theta_i$, the complement of $\Sigma_{-e_N, \bar \theta_i}$ is $\Sigma_{e_N, \theta_i}$, with $\pi> \theta_2>\theta_1>\pi/2$.
Let us consider the solution to 
\begin{equation}\label{def homog}
\begin{cases}
(-\Delta)^s v = 0 & \text{in $\Sigma_{e_N, \theta_2}$}, \\
v >0 &  \text{in $\Sigma_{e_N, \theta_2}$}, \\
v = 0 &  \text{in $\R^N \setminus \Sigma_{e_N, \theta_2}$}. 
\end{cases}
\end{equation}
Existence and uniqueness of $v$, up to a multiplicative constant, are proved in \cite[Theorem 3.2]{BanBog}. It is also proved that $v$ is $\alpha$-homogeneous for some $\alpha>0$ depending on $\theta_2$ and on $s$, and that the exponent $\alpha$ is non-increasing in $\theta_2$, i.e. non-decreasing in $\theta$ (see Lemma \cite[Lemma 3.3]{BanBog}). Since $\theta_2>\pi/2$, the cone $\Sigma_{e_N, \theta_2}$ contains $\R^N_+$, and by Lemma 3.3 and Example 3.2 in \cite{BanBog} we deduce that $\alpha < s$. 

\begin{remark}\label{rem: on monotonicity in theta}
The homogeneity exponent $\alpha$ appearing here is exactly the regularity index $\alpha$ in the thesis of Theorem \ref{thm: boundary regularity}. Therefore, as observed above, the monotonicity of $\alpha$ with respect to $\theta$ is a direct consequence of Lemma 3.3 in \cite{BanBog}.
\end{remark}

Interior regularity theory ensures that $v \in \mathcal{C}^\infty(\Sigma_{e_N, \theta_2})$. Thus, as $\theta_2 > \theta_1$ and $v$ is positive, the restriction of $v$ on $\Sigma_{e_N,\theta_1} \cap \S^{N-1}$ is bounded from below and from above by positive constants. By homogeneity, and recalling that $v$ is uniquely determined up to a  multiplicative constant, we can suppose that there exists $C_0>1$ such that
\begin{equation}\label{continuity estimate}\begin{split}
& v(x)\ge|x|^\alpha \qquad \text{for every $x \in \Sigma_{e_N,\theta_1}$}\\
{\mbox{and }}\quad&
v(x) \le C_0 |x|^\alpha \qquad \text{for every $x \in \R^N$}.
\end{split}\end{equation}
Notice that in this way the choice of $v$ depends only on $\theta$
(recall that~$\theta_1=\pi-\bar \theta_1=\pi-\theta$).
With this notation, we can now prove the following
result:

\begin{lemma}\label{lem: barrier 1}
For $R>0$, let us define
\begin{equation}\label{VUOI:1}
z_R:= 2R^{-\alpha} v - v^2.
\end{equation}
Then, there exists $C>0$ depending only on $\theta$ such that
\[
(-\Delta)^s z_R  \ge  C |x|^{2\alpha-2s} \qquad \text{in $\Sigma_{e_N, \theta_1}$}.
\]
\end{lemma}

\begin{proof}
In light of Lemma \ref{lem: frac Euler}, we can compute (notice that
we omit the constant $c_{N,s}$ and the principal value sense to simplify the notation)
\begin{equation}\label{stima su delta}
\begin{split}
-(-\Delta)^s z_R (x) & = - 2 R^{-\alpha} \underbrace{(-\Delta)^s v (x)}_{=0} + (-\Delta)^s (v^2)(x) \\
& =  \int_{\R^N} \frac{v^2(x) - v^2 (y)}{|x-y|^{N+2s}}\,dy \\
& = v(x) \underbrace{ \int_{\R^N} \frac{v(x) - v (y)}{|x-y|^{N+2s}}\,dy}_{=(-\Delta)^s v(x) = 0} + \int_{\R^N} v(y) \frac{v(x) - v (y)}{|x-y|^{N+2s}}\,dy \\
& =  \int_{\R^N} v(y) \frac{v(x) - v (y)}{|x-y|^{N+2s}}\,dy - v(x) \underbrace{\int_{\R^N} 
\frac{v(x) - v (y)}{|x-y|^{N+2s}}\,dy}_{=(-\Delta)^s v(x) = 0} \\
& = -  \int_{\R^N} \frac{(v(x) - v (y))^2}{|x-y|^{N+2s}}\,dy \le - C |x|^{2\alpha -2s}
\end{split}
\end{equation}
for any $x \in \Sigma_{e_N, \theta_1}$, with $C>0$ depending only on $\theta$
and $v$. This gives the desired result.
\end{proof}

Since $\alpha <s$, Lemma~\ref{lem: barrier 1}
suggests that $z_R$ can be used to construct a wall of upper barriers for $u$ in $\Omega \cap B_R$. 
This is indeed exactly what happens in the classical
case, but
a similar statement does not hold here,
due to the nonlocal nature of the problem. Indeed, in order to apply the 
comparison principle, we should control the sign of the function~$z_R-u$ in the whole 
complement of $\Omega \cap B_R$ and not only along its boundary. On the other hand,
one cannot conclude directly in our case
that $z_R-u \ge 0$ in the
complement of $\Omega \cap B_R$, since by~\eqref{continuity estimate} the function~$z_R$ 
is negative if $x \in \Sigma_{e_N,\theta_2}$ and $|x|$ is very large. 
To overcome this problem, in what follows we define a suitable truncation $w_R$ of $z_R$, being careful enough so that the thesis of Lemma \ref{lem: barrier 1} still holds for $w_R$, and moreover $w_R-u \ge 0$ in $\R^N \setminus (\Omega \cap B_R)$. 

To this purpose, let us consider the algebraic equation
\begin{equation}\label{0pojAyA990:1} 2R^{-\alpha} t - t^2 = 1 \end{equation}
which possesses the two solutions
\begin{equation}\label{0pojAyA990:2}  t_{1}= R^{-\alpha} -\sqrt{R^{-2\alpha} -1}\quad{\mbox{ and }}\quad
t_{2}= R^{-\alpha} + \sqrt{R^{-2\alpha} -1},\end{equation}
provided that~$R\le1$.
We introduce the set 
\begin{equation}\label{9A78aA}
D_R:= \{x \in \R^N: v(x) \ge t_2\} = \left\{x \in \R^N: 
v(x) \ge R^{-\alpha} + \sqrt{R^{-2\alpha} -1}\right\}
.\end{equation}
Notice that, by \eqref{continuity estimate}, $D_R \neq \varnothing$. Moreover, if~$R>0$ is sufficiently small, we have that
\begin{equation}\label{D_R exterior}
{\mbox{if }}
|x| \le \left(\frac{3}{2C_0}\right)^{1/\alpha} \frac{1}{R} 
{\mbox{ then }}
x \not \in D_R
\end{equation}
and so, in particular,
\begin{equation}\label{D_R exterior:2}
D_R\subseteq\R^N\setminus B_R.
\end{equation}
Now we construct the barrier needed for our purposes:

\begin{lemma}\label{lem: barrier 2} Let~$z_R$ be the function introduced
in Lemma~\ref{lem: barrier 1}.

Let 
\begin{equation}\label{VUOI:2}
w_R(x):=\begin{cases}  z_R(x) & \text{if $x \in \R^N \setminus D_R$}, \\  
1 & \text{if $x \in D_R$}.
\end{cases}
\end{equation}
Then, there exist $\bar R >0$ small enough and a constant $\bar C>0$, both depending only on $\theta$, such that 
\begin{equation}\label{9ikAHHKA000}
\begin{cases}
(-\Delta)^s w_R \ge \bar C |x|^{2\alpha-2s} & \text{in $\Sigma_{e_N,\theta_1} \cap B_R$}, \\
w_R  \ge 0 & \text{in $\R^N \setminus \Sigma_{e_N,\theta_1}$},\\ 
w_R \ge 1 & \text{in $\Sigma_{e_N,\theta_1} \setminus B_R$},
\end{cases}
\end{equation}
for every $R \in (0,\bar R]$.
\end{lemma}

\begin{proof} First of all, we observe that
\begin{equation}\label{FO wR:001}
{\mbox{$w_R \ge 0$ in $\R^N$.}}\end{equation} 
Indeed, if $w_R(x) <0$ then necessarily  
\begin{equation}\label{2604}
x \not \in D_R \quad \text{and} \quad 2R^{-\alpha}v(x) -v^2(x) <0.
\end{equation}
The second inequality is satisfied if and only if $v(x)>2R^{-\alpha}$. But $2R^{-\alpha}>t_2$, so that $v(x)>2R^{-\alpha}$ implies $x \in D_R$. This shows that the two conditions in \eqref{2604} cannot take place simultaneously, and proves~\eqref{FO wR:001}.

Now we recall~\eqref{D_R exterior:2} and we write
\[
\Sigma_{e_N,\theta_1} \setminus B_R = (\Sigma_{e_N,\theta_1} \cap D_R)
\cup \big(\Sigma_{e_N,\theta_1} \setminus (D_R \cup B_R)\big).
\] 
We claim that
\begin{equation}\label{56789ojhgfcvA}
{\mbox{$w_R \ge 1$ in $\Sigma_{e_N,\theta_1} \setminus B_R$.}}
\end{equation}
To this aim, we point out that~$\Sigma_{e_N,\theta_1} \cap D_R\subseteq
\Sigma_{e_N,\theta_1} \setminus B_R$, thanks to~\eqref{D_R exterior:2}, and so, by~\eqref{9ikAHHKA000},
we obtain that~$w_R \ge 1$ in $\Sigma_{e_N,\theta_1} \cap D_R$.
So, to complete the proof of~\eqref{56789ojhgfcvA},
we focus now on the value of~$w_R$ on the points of~$\Sigma_{e_N,\theta_1} \setminus (D_R \cup B_R)$.
For this, we recall~\eqref{0pojAyA990:1}, \eqref{0pojAyA990:2} and~\eqref{9A78aA},
and we observe that
\begin{equation}\label{insieme cattivo}
\begin{split}
(\R^N \setminus D_R) \cap \{w_R < 1\} &= (\R^N \setminus D_R) \cap (\{ v < t_1\} \cup \{v > t_2\}) \\
&= (\R^N \setminus D_R) \cap \{ v < t_1\} \\&=  \{ v < t_1\}.
\end{split}
\end{equation}
Now we remark that
\begin{equation}\label{5678hgfvbghnAAAA}
\{ v < t_1\}\subseteq B_R.
\end{equation}
Indeed,
if~$v(x) < t_1 = R^{-\alpha} -\sqrt{R^{-2\alpha}-1}$, then by \eqref{continuity estimate}
we know that~$|x|^{\alpha} < t_1$, i.e.
$$ |x| < \left( R^{-\alpha} -\sqrt{R^{-2\alpha}-1} \right)^{1/\alpha}
<R,$$
which proves~\eqref{5678hgfvbghnAAAA}.

{F}rom~\eqref{insieme cattivo} and~\eqref{5678hgfvbghnAAAA}, we deduce that
$$ (\R^N \setminus D_R) \cap \{w_R < 1\} \subseteq B_R.$$
As a consequence,
\[ \Sigma_{e_N,\theta_1} \setminus (D_R \cup B_R)] \cap \{w_R <1\} = \varnothing,
\]
and therefore~$w_R \ge 1$ in $\Sigma_{e_N,\theta_1} \setminus (D_R \cup B_R)$
for every $R$ sufficiently small, which completes the proof of~\eqref{56789ojhgfcvA}.

Now we focus on
the inequality satisfied by the fractional Laplacian of~$w_R$ in $\Sigma_{e_N,\theta_1} \cap B_R$.
We first observe that if~$x \in \Sigma_{e_N,\theta_1} \cap B_R$, then~$x\in \Sigma_{e_N,\theta_1}
\setminus D_R$, thanks to~\eqref{D_R exterior:2}, and so~$w_R(x)=z_R(x)$.
Using this and
Lemma \ref{lem: barrier 1}, we obtain that,
for any~$x \in \Sigma_{e_N,\theta_1} \cap B_R$,
\begin{equation}\label{yhikmtgbHNk}\begin{split}
-(-\Delta)^s w_R(x)\, & =  -(-\Delta)^s z_R(x) + \big[(-\Delta)^s z_R(x) -(-\Delta)^s w_R(x)\big] \\
& \le - C |x|^{2\alpha-2s} + \int_{\R^N} \frac{z_R(x)- z_R(y) -w_R(x) + w_R(y)}{|x-y|^{N+2s}}\,dy \\
& = - C |x|^{2\alpha-2s} + \int_{D_R} \frac{w_R(y)- z_R(y)}{|x-y|^{N+2s}}\,dy \\
& = - C |x|^{2\alpha-2s} + \int_{D_R} \frac{ 1 - z_R(y)}{|x-y|^{N+2s}}\,dy \\
& \le - C |x|^{2\alpha-2s} + \int_{\R^N \setminus B_{C_1/R}} \frac{1 + |z_R(y)|}{|x-y|^{N+2s}}\,dy
\end{split}\end{equation}
where in the last inequality we used \eqref{D_R exterior} with $C_1:=(3/(2C_0))^{1/\alpha}$. 

Notice also that,
if~$y \in \R^N \setminus B_{C_1/R}$, then
\[
R^{-\alpha} = C_1^{-\alpha} \left( \frac{C_1}{R}\right)^{\alpha} \le C_1^{-\alpha} |y|^\alpha
.\]
Consequently, by \eqref{continuity estimate},
\begin{equation}\label{332}
|z_R(y)| \le 2 R^{-\alpha} v(y) + v^2 (y)\le C |y|^{2\alpha}  \qquad \text{for any $y \in \R^N \setminus B_{C_1/R}$}.
\end{equation}
Furthermore, since $R>0$ is appropriately small, we have
\begin{equation}\label{333}
|x-y| \ge |y|- |x| \ge |y|- R \ge \frac{|y|}{2} \qquad \text{for $x \in B_R$ and $y \in \R^N \setminus B_{C_1/R}$}.
\end{equation}
Plugging~\eqref{332} and \eqref{333} into~\eqref{yhikmtgbHNk}, we infer that
\begin{align*}
-(-\Delta)^s w_R(x)  & \le - C |x|^{2\alpha-2s} +  \int_{\R^N \setminus B_{C_1/R}} \frac{1 + |z_R(y)|}{|x-y|^{N+2s}}\,dy \\
& \le - C |x|^{2\alpha-2s} + C \int_{\R^N \setminus B_{C_1/R}} \frac{|y|^{2\alpha}}{|y|^{N+2s}}\,dy \\
& \le - C  |x|^{2\alpha-2s} + C \int_{C_1/R}^{+\infty} r^{2\alpha-2s-1}\,dr \\
& =  - C  |x|^{2\alpha-2s} + C R^{2s-2\alpha} 
\end{align*}
for any~$x\in\Sigma_{e_N,\theta_1} \cap B_R$.
Hence, recalling that~$R>0$ is sufficiently small, that~$|x|<R$, and that~$\alpha<s$,
we obtain that~$
-(-\Delta)^s w_R(x) \le - \bar C  |x|^{2\alpha-2s}$,
for any~$x\in\Sigma_{e_N,\theta_1} \cap B_R$,
for an appropriate~$\bar C>0$, as desired.
\end{proof}

Now we are in the position of completing the proof of Proposition \ref{prop: 4.1 in BCN}.

\begin{proof}[Proof of Proposition \ref{prop: 4.1 in BCN}]
Let 
\[
\tilde u:= \frac{u}{\max\left\{ \|u\|_{L^\infty(\R^N)} ,1 \right\} }.
\]
Notice that
\begin{equation}\label{adgfhjIUJBAjHGFGHJA}
{\mbox{$\tilde u\le1$ in~$\R^N$.}}\end{equation}
Let also~$w_R$ be the function introduced in Lemma~\ref{lem: barrier 2}.
We claim that
\begin{equation}\label{adgfhjIUJBAjHGFGHJA:BIS}
{\mbox{$w_R \ge \tilde u$ in $\Omega \cap B_R$}}\end{equation}
for $R \in (0, \bar R]$ (where~$\bar R>0$ is small enough, as given
by Lemma~\ref{lem: barrier 2}).
To prove~\eqref{adgfhjIUJBAjHGFGHJA:BIS}, we observe that, by
Lemma \ref{lem: barrier 2} 
and~\eqref{adgfhjIUJBAjHGFGHJA},
we know that 
\begin{equation}\label{adgfhjIUJBAjHGFGHJA:TRIS}
{\mbox{$w_R \ge \tilde u$ in $\R^N \setminus (\Omega \cap B_R)$.}}\end{equation}
Also, we set
\begin{equation}\label{def R2}
R := \min\left\{ \frac{\bar R}{2}, \left( \frac{\bar C \max\left\{ \|u\|_{L^\infty(\R^N)} ,1 \right\}}{2 \|g\|_{L^\infty(\Omega)}} \right)^{\frac{1}{2s-2\alpha}}\right\}.
\end{equation}
Notice in particular that $R < \bar R$.
Hence, using again Lemma \ref{lem: barrier 2} and~\eqref{p_r}, we have that
\begin{align*}
(-\Delta)^s (w_R-\tilde u)(x) \,& \ge \bar C |x|^{2\alpha-2s} - \frac{\|g\|_{L^\infty(\Omega)}}{\max\left\{ \|u\|_{L^\infty(\R^N)} ,1 \right\}} \\
& \ge |x|^{2\alpha-2s} \left( \bar C - R^{2s-2\alpha} \frac{\|g\|_{L^\infty(\Omega)}}{\max\left\{ \|u\|_{L^\infty(\R^N)} ,1 \right\}}\right) \\
& \ge \frac{\bar C}{2} |x|^{2\alpha-2s} \ge 0
\end{align*}
for any~$x\in\Omega \cap B_R$.

This, \eqref{adgfhjIUJBAjHGFGHJA:TRIS}
and the maximum principle, imply that $w_R \ge \tilde u$ in $\Omega \cap B_R$, i.e.
\begin{equation}\label{233:0}
u(x) \le (1+\|u\|_{L^\infty(\R^N)}) w_R(x)
\end{equation}
for any~$x\in\Omega \cap B_R$. 
On the other hand, if~$x\in\Omega \cap B_R$, we deduce from~\eqref{D_R exterior:2}
that~$x\in \Omega\setminus D_R$, and therefore, by~\eqref{VUOI:1} and~\eqref{VUOI:2},
$$ w_R(x)\le 2R^{-\alpha} v(x).$$
By inserting this into~\eqref{233:0}, we conclude that
\begin{equation}\label{233}
u(x) \le 2 (1+\|u\|_{L^\infty(\R^N)}) R^{-\alpha} v(x)\end{equation}
for any~$x\in\Omega \cap B_R$.
Now we observe that, by~\eqref{def R2},
\begin{equation}\label{stima su R}
R^{-\alpha}  \le \left( \frac{2}{\bar R}\right)^\alpha + \left(\frac{2 \|g\|_{L^\infty(\Omega)}}{\bar C \max\left\{ \|u\|_{L^\infty(\R^N)} ,1 \right\}} \right)^{\frac{\alpha}{2s-2\alpha}}
\le C \left( 1+  \|g\|_{L^\infty(\Omega)}^{\frac{\alpha}{2s-2\alpha}}\right),
\end{equation}
with $C$ depending only on $\theta$ (recall that $\bar R$ depends only on $\theta$). Therefore,
recalling~\eqref{233} and~\eqref{continuity estimate}, we conclude that
\[
u(x) \le  C\,\big(1+\|u\|_{L^\infty(\R^N)}\big)\, \left( 1+  \|g\|_{L^\infty(\Omega)}^{\frac{\alpha}{2s-2\alpha}}\right)\, |x|^\alpha,
\]
for any~$x\in\Omega \cap B_R$,
for some positive constant $C$ depending only on $\theta$.

The same argument can be repeated replacing $u$ with $-u$, and so we conclude that,
for any~$x\in\Omega \cap B_R$,
\begin{equation}\label{456789oJHGAGGkkk} |u(x)|
\le  C\,\big(1+\|u\|_{L^\infty(\R^N)}\big)\, 
\left( 1+  \|g\|_{L^\infty(\Omega)}^{\frac{\alpha}{2s-2\alpha}}\right)\, |x|^\alpha .\end{equation}
Now, take~$p\in\Omega$. We distinguish two cases,
either~$\dist(p,\pa \Omega)<R$ or~$\dist(p,\pa \Omega)\ge R$.

If~$\dist(p,\pa \Omega)<R$, we consider~$\tilde p\in\pa\Omega$ to be
a projection of~$p$ along~$\pa\Omega$. 
Up to a rigid motion, we may also suppose that~$\tilde p=0$
(hence we are in the normalized setting of~\eqref{78909uyHJAHH678911}).
In this way, we have that
$$  R>\dist(p,\pa \Omega)=|p-\tilde p|=|p|.$$
Hence, from~\eqref{456789oJHGAGGkkk}, we have that
\begin{align*}
 |u(p)|
& \le  C\,\big(1+\|u\|_{L^\infty(\R^N)}\big)\,
\left( 1+  \|g\|_{L^\infty(\Omega)}^{\frac{\alpha}{2s-2\alpha}}\right)\, |p|^\alpha \\
& =C\,\big(1+\|u\|_{L^\infty(\R^N)}\big)\,
\left( 1+  \|g\|_{L^\infty(\Omega)}^{\frac{\alpha}{2s-2\alpha}}\right)\, 
\dist(p,\pa \Omega)^\alpha,
\end{align*}
which establishes~\eqref{7899oiHYGFAGHJhgfghj11OA} in this case.

If instead~$\dist(p,\pa \Omega)\ge R$, we have that
\[
|u(p)| \le \|u\|_{L^\infty(\R^N)} = \|u\|_{L^\infty(\R^N)} \frac{\dist(p,\pa \Omega)^\alpha}{\dist(p,\pa \Omega)^\alpha} \le \frac{\|u\|_{L^\infty(\R^N)} }{R^\alpha} \dist(p,
\pa \Omega)^\alpha.
\]
Then, the estimate in~\eqref{7899oiHYGFAGHJhgfghj11OA} follows in this case from~\eqref{stima su R}.
This concludes the proof of Proposition~\ref{prop: 4.1 in BCN}.
\end{proof}

Thanks to the above results, we can now complete the proof of Theorem \ref{thm: boundary regularity}:

\begin{proof}[Proof of Theorem \ref{thm: boundary regularity}]
Since $u \in L^\infty(\R^N)$, we can
restrict to the case when $|x-y|$ is small, say \begin{equation}\label{5789yHA0987656tyaa}
|x-y|< \frac{R}{4},\end{equation} 
with $R>0$ defined in \eqref{def R2}. It is also not restrictive to suppose that $R <2$.
We prove the thesis of the theorem according to four different cases. \medskip

\emph{Case 1)} Assume that
$$x,y \in \Omega {\mbox{ with }}
\dist(x,\partial \Omega) \ge \frac{R}2.$$ 
In this case, we have that $B_{R/2}(x) \subset \Omega$.
In particular, by~\eqref{5789yHA0987656tyaa},
$$ x, y\in B_{R/2}(x) \subset \Omega.$$
Therefore, by scaling the interior H\"older estimate in \cite[Proposition 2.9]{SilCPAM},
we deduce that for a constant $C>0$ depending only on $\theta$
$$ |u(x)-u(y)| 
\le C\left[ \left( \frac{2}{R}\right)^\alpha \|u\|_{L^\infty(\R^N)} +  \left( \frac{R}{2}\right)^{2s -\alpha} \|g\|_{L^\infty(\Omega)}\right]  |x-y|^\alpha .$$
Now, using \eqref{stima su R}, we obtain
$$ |u(x)-u(y)|
\le
C\,\left[ (1+\|g\|_{L^\infty(\Omega)}^{\frac{\alpha}{2s-2\alpha}}) \|u\|_{L^\infty(\R^N)} +   \|g\|_{L^\infty(\Omega)}\right]  |x-y|^\alpha,$$
which is the desired result in this case. \medskip

\emph{Case 2)} Now assume that
$$x,y \in \Omega
{\mbox{ with }}
\dist(x,\partial \Omega) \le d(y,\pa \Omega) \le\frac{ R}2
{\mbox{ and }} |x-y| \ge \frac{\dist(y,\pa \Omega)}8.$$ Then, by Proposition \ref{prop: 4.1 in BCN},
\begin{align*}
|u(x)-u(y)| &\le 2\max\{|u(x)|,|u(y)|\}  \\
& \le C\,\big(1+ \|u\|_{L^\infty(\R^N)}\big)\, \left(
1+\|g\|_{L^\infty(\Omega)}^{\frac{\alpha}{2s-2\alpha}}\right)\, \dist(y,\pa \Omega)^\alpha \\
&  \le C 8^\alpha C\,\big(1+ \|u\|_{L^\infty(\R^N)}\big)\, \left(
1+\|g\|_{L^\infty(\Omega)}^{\frac{\alpha}{2s-2\alpha}}\right)\,  |x-y|^\alpha,
\end{align*}
as desired.
\medskip

\emph{Case 3)} Now we suppose that
$$x,y \in \Omega {\mbox{ with }} \dist(x,\partial \Omega)\le\dist(y,\pa \Omega) \le \frac{R}2
{\mbox{ and }} 
|x-y| \le \frac{\dist(y,\pa \Omega)}8.$$ Let us set 
$$\rho:= \frac{\dist(y,\pa \Omega)}4$$
and $\tilde u(z):= u(y+\rho z)$. 

We remark that if~$z\in B_4$ then
$$ y+\rho z\in B_{4\rho}(y)=B_{\dist(y,\pa \Omega)}(y)\subseteq\Omega.$$
Hence, for any~$z\in B_4$, we obtain from~\eqref{p_r} that
$$ (-\Delta)^s \tilde u(z) = \rho^{2s} g(y+\rho z).$$
Accordingly, using the fact that $\alpha <s$, we obtain that, for any~$z\in B_4$,
\begin{equation}\label{2130}
|(-\Delta)^s \tilde u(z)| = |\rho^{2s} g(y+\rho z)| \le \rho^\alpha \|g\|_{L^\infty(\Omega)}.
\end{equation}
Moreover, thanks to Proposition \ref{prop: 4.1 in BCN}, we have that
\begin{align*}
|\tilde u(z)|  & =  |u(y+\rho z)| \le   C\,
\big(1+ \|u\|_{L^\infty(\R^N)}\big)\, \left(
1+\|g\|_{L^\infty(\Omega)}^{\frac{\alpha}{2s-2\alpha}}\right)\,
\dist(y+\rho z,\pa \Omega)^\alpha \\
& \le C\,\big(1+ \|u\|_{L^\infty(\R^N)}\big)\, \left(1+\|g\|_{L^\infty(\Omega)}^{\frac{\alpha}{2s-2\alpha}} \right)\,\left( \dist(y,\pa \Omega)+ \rho |z|\right)^\alpha  \\
& \le  C\,\big(1+ \|u\|_{L^\infty(\R^N)}\big)\, \left(1+\|g\|_{L^\infty
(\Omega)}^{\frac{\alpha}{2s-2\alpha}} \right)\, \rho^\alpha \,( 1 + |z|)^\alpha
\end{align*}
for any $z \in \R^N$, up to renaming~$C>0$.

As a consequence,
\begin{equation}\label{2131}\begin{split}
& \sup_{B_2} |\tilde u|  \le C\,\big(1+ \|u\|_{L^\infty(\R^N)}\big)\, \left(
1+\|g\|_{L^\infty(\Omega)}^{\frac{\alpha}{2s-2\alpha}}\right)\, \rho^\alpha
\\ {\mbox{and }}\qquad&
\int_{\R^N} \frac{|\tilde u(z)|}{(1+|z|)^{N+2s}}\,dz \le C\,
\big(1+ \|u\|_{L^\infty(\R^N)}\big)\, \left(
1+\|g\|_{L^\infty(\Omega)}^{\frac{\alpha}{2s-2\alpha}}\right)  \rho^\alpha.
\end{split}\end{equation}
Estimates \eqref{2130} and~\eqref{2131} allow us to apply Corollary 2.5 in \cite{RosSerJMPA},
which implies that
\begin{equation}\label{8ih3456789ytrdfghjjhgfd}
\|\tilde u\|_{\mathcal{C}^{0,\alpha}(\overline{B_{1/2}})} \le C \left[ \big(
1+ \|u\|_{L^\infty(\R^N)}\big)\, \left(
1+\|g\|_{L^\infty(\Omega)}^{\frac{\alpha}{2s-2\alpha}}\right) + \|g\|_{L^\infty(\Omega)}\right] \, \rho^\alpha.
\end{equation}
Now we observe that
$$ \frac{|x-y|}{\rho} \le \frac{\dist(y,\pa \Omega) }{8\rho} =\frac{1}{2}$$
and so
$$ z_* := \frac{x-y}{\rho} \in \overline{B_{1/2}}.$$
As a consequence, we deduce from~\eqref{8ih3456789ytrdfghjjhgfd} that
\begin{align*}
|u(x) - u(y)| &=|\tilde u(z_*)-\tilde u(0)|\\
& \le C\left[  \big(1+ \|u\|_{L^\infty(\R^N)}\big)\,\left(
1+\|g\|_{L^\infty(\Omega)}^{\frac{\alpha}{2s-2\alpha}} \right) + \|g\|_{L^\infty(\Omega)} \right]\,\rho^\alpha \,
\left( \frac{|x-y|}{\rho}\right)^\alpha,
\end{align*}
which gives the desired result in this case.
\medskip

\emph{Case 4)} Finally, we consider the case in which
$$x \in \Omega
{\mbox{ and }} 
y \in \R^N \setminus \Omega.$$ By Proposition \ref{prop: 4.1 in BCN} 
\[
|u(x)-u(y)| = |u(x)| \le C\,  \big(1+ \|u\|_{L^\infty(\R^N)}\big)\, \left(
1+\|g\|_{L^\infty(\Omega)}^{\frac{\alpha}{2s-2\alpha}} \right)\,\dist(x,\pa \Omega)^\alpha.
\]
Since in this case~$\dist(x,\pa \Omega) < |x-y|$, this completes the proof.
\end{proof}

\section{Monotonicity and qualitative properties in globally Lipschitz epigraphs}\label{sec: thm 1}

In this section, we give
the proof of Theorem \ref{thm: main 1}.
Differently from the strategy adopted in \cite{BCNCPAM},
in our case property ($iv$) is a direct consequence of 
the general result in Theorem \ref{thm: boundary regularity}.

The organization of this section is the following.
In Subsection \ref{sec: convergence} we prove properties ($i$) and ($ii$) in 
Theorem \ref{thm: main 1}. Property ($iii$) is the object of Subsection~\ref{sec: growth}. The uniqueness and the monotonicity in $x_N$ are proved in Subsection \ref{sec: uniqueness} with a unified approach, simplifying the proof in \cite{BCNCPAM}.
Finally, the general monotonicity property in point ($vi$) of Theorem \ref{thm: main 1}
follows simply by a 
suitable rotation of coordinates. We point out that, for such argument, the global Lipschitz continuity of the
epigraph $\Omega$ is needed. 

\medskip

Before proceeding, we observe that it is not restrictive to suppose from now on that
\begin{equation}\label{mu=1}
\mu=1 \quad \text{in assumptions ($f1$)-($f3$)},
\end{equation} 
for the sake of simplicity. Moreover, we define
\begin{equation}\label{M:M} M:=\sup_\Omega u.\end{equation}
In this way, we have that~$0 < u  \le  M$ in $\Omega$. Accordingly,
in~\eqref{problem} only the restriction of $f$ on the interval $[0,M]$ plays a role. 
Therefore, we can modify the definition of $f$ outside $[0,M]$ without changing the 
equation, and as a consequence from now on we can suppose that 
\begin{equation}\label{rem: f Lip}
\text{$f$ is not only locally Lipschitz, but globally Lipschitz continuous},
\end{equation}
and we denote by $L$ its Lipschitz constant.

\subsection{Uniform convergence of $u$ to $1$}\label{sec: convergence}

Our goal now is to show that $u < 1$ in $\Omega$, as claimed
in Theorem \ref{thm: main 1}-($i$)
(and hence, comparing with~\eqref{M:M}, it follows that $M \le 1$). 

\begin{proof}[Proof of Theorem \ref{thm: main 1}-($i$)]
In order to show that $u \le 1$ we can proceed as in \cite{BCNCPAM}, 
using Theorem \ref{thm: super max}. For the strict inequality it is sufficient to observe that, being $f(1)=0$, we are in position to apply 
the strong maximum principle to the function $1-u$. For further details, compare with~\cite[page 1095]{BCNCPAM}.
\end{proof}

In the rest of the subsection we prove property ($ii$) in Theorem \ref{thm: main 1}
(the local counterpart of this strategy is contained in Section 3 in \cite{BCNCPAM}).

\begin{lemma}\label{lem: 3.1 BCN}
Let $D$ be an open subset of $\R^N$, and let 
$g$ be a locally Lipschitz continuous function.

Let~$v$ be a classical supersolution  to
\[
\begin{cases}
(-\Delta)^s v \ge g(v) & \text{in $D$}, \\
v >0 & \text{in $D$}, \\
v \ge 0 & \text{in $\R^N \setminus D$}.
\end{cases}
\]
Let also~$B$ be a ball with closure $\overline{B}$ contained in $D$, and let $z$ be a classical subsolution to 
\[
\begin{cases}
(-\Delta)^s z \le g(z) & \text{in $B \cap \{z>0\}$}, \\
z \le v & \text{in $B$}, \\
z \le 0 & \text{in $\R^N \setminus B$}.
\end{cases}
\]
Then for any one-parameter continuous family of Euclidean motions $\{A(t): 0 \le t \le T\}$ with $A(0)= \textrm{Id}$ and $A(t) \overline{B} \subset D$ for every $t$, it results that
\[
z_t(x):= z(A(t)^{-1} x) < u(x) \qquad \text{in $B_t:= A(t) B$},
\]
for every $t \in [0,T]$.
\end{lemma}

The proof of Lemma~\ref{lem: 3.1 BCN}
is very similar to the one of \cite[Lemma 3.1]{BCNCPAM} and thus is omitted. It uses the fact that $(-\Delta)^s$ is invariant under rigid motions and the strong maximum principle.


Exploiting Lemma \ref{lem: 3.1 BCN} we can deduce a lower estimate for $u$ far away from the boundary $\partial \Omega$.

\begin{lemma}\label{lem: 3.2 BCN}
There exist $\eps_1$, $R_0>0$ with $R_0$ depending only on $N$ and $\delta_0$ (recall assumption ($f2$)) such that
\[
u(x) > \eps_1 \qquad \text{if $\dist(x,\partial \Omega) > R_0$}.
\]
\end{lemma}

The proof of Lemma~\ref{lem: 3.2 BCN} here is a simple extension to 
that in \cite[Lemma 3.2]{BCNCPAM} and therefore is omitted. 

We stress that if $\lambda_1(B_R)$ denotes the first eigenvalue of $(-\Delta)^s$ in~$B_R$
with homogeneous Dirichlet boundary condition, then
\[
\lambda_1(B_R) \to 0 \qquad \text{as $R \to +\infty$},
\]
by scaling.

We are now ready to prove the counterpart of Lemma 3.3 in \cite{BCNCPAM}. 
We observe that in this framework the ``local" argument cannot be directly extended, 
since the proof 
of Lemma 3.3 in \cite{BCNCPAM} heavily relies on
local properties of functions whose Laplacian has a sign, 
and this technique does not work in a nonlocal setting. 
Therefore, to overcome such difficulty,
we have to modify the approach in the following way.

\begin{lemma}\label{lem: 3.3 BCN}
Let $\eps_1$, $R_0>0$ be given by Lemma~\ref{lem: 3.2 BCN}.
Let~$y \in \Omega$ with $\dist(y,\partial \Omega)>R_0$, so that $u(y) \ge \eps_1$. Let $\eps>0$ so small that $(1+\eps)u(y) <1$. Let
\[
 \delta_{\eps,y} := \min \left\{ f(t): \;t \in [\eps_1,(1+\eps) u(y)] \right\}>0.
\]
Then, there exists $C_1>0$ depending only on $s$ and $N$ such that 
\[
C_1 \delta_{\eps,y} \le \left[ \dist(y,\partial \Omega)-R_0 \right]^{-2s}.
\]
\end{lemma}
 
\begin{remark}
Notice that the existence of $\eps$ is guaranteed by the strict upper estimate $u <1$ 
in $\Omega$, thanks to 
Theorem \ref{thm: main 1}-($i$). The crucial fact in the lemma is that $C_1$ does not depend on $y$ 
and~$\eps$.
\end{remark}

\begin{proof}[Proof of Lemma \ref{lem: 3.3 BCN}]
Let $v$ be the solution to 
\[
\begin{cases}
(-\Delta)^s v=1 & \text{in $B_1$}, \\
v = 0 & \text{in $\R^N \setminus B_1$}.
\end{cases}
\]
We claim that the thesis is true with $C_1:= \max_{B_1} v = v(0)$. If not, we have
\[
C_1 \delta_{\eps,y} > \left[ \dist(y,\partial \Omega)-R_0 \right]^{-2s},
\]
and hence there exists $R>0$ such that
\begin{equation}\label{def R}
C_1 \delta_{\eps,y} > \frac{1}{R^{2s}} > \left[ \dist(y,\partial \Omega)-R_0 \right]^{-2s}.
\end{equation}
Notice in particular that $R+R_0< \dist (y,\partial \Omega)$, and so
\begin{equation}\label{INA:A}
B_R(y) \subset \{ \dist (x,\partial \Omega) > R_0\}.\end{equation} Let 
\[
z(x):= R^{2s} \delta_{\eps,y} v\left( \frac{x-y}{R}\right).\]
Then, by scaling, we have that
\[ \begin{cases}
(-\Delta)^s z=\delta_{\eps,y} & \text{in $B_R(y)$}, \\
z = 0 & \text{in $\R^N \setminus B_R(y)$},
\end{cases}
\]
and the maximum of $z$ is $z(y)= R^{2s} \delta_{\eps,y} C_1$. 

By \eqref{INA:A}, we are now in the position of exploiting
Lemma \ref{lem: 3.2 BCN}: in this way,
we have $\tau z < u$ in $B_R(y)$ provided $\tau >0$ is sufficiently small. 
Now we increase $\tau$ till we obtain a touching point, namely we
let 
\[
\bar \tau:= \inf\{\tau>0:  \tau z(x_0) = u(x_0) \text{ for some $x_0 \in \overline{B_R(y)}$}\}.
\]
As a matter of fact, since~$z=0<u$ on~$\partial B_R(y)$,
we have that~$x_0$ lies in the interior of the ball~$ B_R(y)$.

By definition, and since $z$ is radial and radially decreasing with respect to $y$, 
\[
u(x_0) = \bar \tau z(x_0) \le\bar  \tau z(y) = \bar \tau C_1 \delta_{\eps,y} R^{2s} \le u(y) <1.
\]
Using this we infer that $u(x_0) \le u(y)$, and moreover $
\bar \tau C_1 \delta_{\eps,y} R^{2s} <1$, which together with \eqref{def R} implies that~$\bar \tau <1$. 

We are ready to complete the contradiction argument: since $u(x_0) \le u(y)$, by continuity there exists a neighbourhood $U$ of $x_0$ such that 
\begin{equation}\label{AKA:1}
u \le (1+\eps)u(y) \qquad \text{in $U$}.
\end{equation}
Also, by Lemma~\ref{lem: 3.2 BCN} and~\eqref{INA:A}, we have that~$u(x_0)>\eps_1$
and so, by continuity, we have that
\begin{equation}\label{AKA:2}{\mbox{$u\ge\eps_1$ in a small 
neighborhood~$U'$ of~$x_0$.}}\end{equation}
Thus, in~$U'':=U\cap U'$ we have that both~\eqref{AKA:1} and~\eqref{AKA:2}
are satisfied.
Therefore, by the definition of $\delta_{\eps,y}$,
\[
(-\Delta)^s u = f(u) \ge \delta_{\eps,y} \qquad \text{in $U''$}.
\]
As a consequence, since $\bar \tau <1$,
\[
\begin{cases}
(-\Delta)^s (\bar \tau z-u) \le \bar \tau \delta_{\eps,y} -\delta_{\eps,y} < 0 & \text{in $U''$,}\\
(\bar \tau z - u) \le 0  &\text{in $\R^N$,} \\
(\bar \tau z-u)(x_0)=0,
\end{cases}
\]
which by the strong maximum principle implies that~$\bar \tau z \equiv u$. 
This is a contradiction with the fact that $u  >0 = \bar \tau z$ on $\Omega \setminus B_R(y)$.
\end{proof}

\begin{proof}[Proof of Theorem \ref{thm: main 1}-($ii$)]
Let us assume by contradiction that there exist a sequence of
points~$\{y_n\} \subset \Omega$ and some $\rho >0$ such that  
\[
\dist(y_n, \partial \Omega) \to +\infty \quad \text{and} \quad u(y_n) \in (\eps_1,1-\rho].
\]
Then we can choose $\eps>0$ independent of $n$ such that $(1+\eps)(1-\rho)<1$ in 
Lemma~\ref{lem: 3.3 BCN}, deducing that
\begin{equation}\label{AK:Apq}
0 \le \min_{[\eps_1,(1+\eps) u(y_n)]} f \le C_1^{-1} \left[ \dist(y_n,\partial \Omega)-R_0 \right]^{-2s}  \to 0
\end{equation}
as $n \to +\infty$. 

On the other hand, by assumption~($f1$),
\[
 \min_{[\eps_1,(1+\eps) u(y_n)]} f \ge  \min_{[\eps_1,(1+\eps) (1-\rho)]} f >0
 \]
and this is in contradiction with~\eqref{AK:Apq}.
\end{proof}

\subsection{Boundary behaviour of the solution: lower estimate}\label{sec: growth}

In this subsection we describe the growth of the solution near the boundary of $\Omega$, proving point ($iii$) of Theorem \ref{thm: main 1}. We point out that the proof is completely different with respect to the one in \cite{BCNCPAM} (proof of Assertion (c) in Theorem 1.2), where an argument based on the construction of an explicit local barrier is used. Dealing with a nonlocal operator, such an approach fails, and we should produce a global barrier on which explicit computations are much more involved. To overcome the problem, we replace the barrier-argument with a convenient geometric constructions, which permits to apply iteratively the Harnack inequality. This approach seems to be applicable to a wide class of operators.

\begin{proof}[Proof of Theorem \ref{thm: main 1}-($iii$)]
Up to a translation, it is not restrictive to suppose that $0 \in \partial \Omega$. We start with three simple consequences of the fact that $\Omega = \{x_N > \varphi(x')\}$ is a globally Lipschitz epigraph (recall that $K$ denotes the Lipschitz constant of $\varphi$).\smallskip
 
1) Let 
$$\Sigma_\beta:= \{|x'|< (\tan{\beta}) x_N\}$$ be an infinite open 
cone, with $0< \beta<\pi/2$. It is possible to choose $\beta$, depending only 
on the Lipschitz constant $K$, so that if the vertex of $\Sigma_\beta$ is 
translated to any point of $\pa \Omega$, then the cone is included in $\Omega$: that is there exists $\bar \beta \in (0,\pi/4)$ such that
\[
x_0 + \Sigma_\beta \subset \Omega \qquad \text{for every $x_0 \in \partial \Omega$ and for every $0 < \beta < \bar \beta$}.
\] 
In what follows we fix $\beta \in (0,\bar \beta)$.\smallskip

2) By the Lipschitz continuity of $\varphi$, and 
recalling that $0 \in \partial \Omega$ (so that $0 = \varphi(0')$),
from Lemma \ref{lem: 3.2 BCN} and Theorem \ref{thm: main 1}-($ii$) it follows the existence of $h_1, h_2>0$ such that 
\begin{equation}\label{5.999}
\begin{split}
& u(x) > \eps_1 \qquad \text{if $x_N > K |x'| + h_1 \ge  \varphi(x')+ h_1$},
\\ {\mbox{and }}\quad&
u(x) > 2\eps_1 \qquad \text{if $x_N > K |x'| + h_2 \ge \varphi(x')+ h_2$}
.\end{split}\end{equation}
Here we implicitly suppose that $\eps_1 < 1/2$; if this were
not the case, we can simply replace $\eps_1$ with a smaller quantity. 
If necessary replacing $h_2$ with a larger quantity, we can also suppose that 
\begin{equation}\label{additional}
\frac{\log \left(\frac{h_2}{1+\sin \beta}\right)}{ \log \left(1+ \tan \beta\right)}- \frac{\log h_1 }{ \log \left(1+ \tan \beta\right)}  >1.
\end{equation}\smallskip

3) Finally we observe that
\[
\{ x_N \le \varphi(x') + h_2\} \cap \Sigma_\beta \subset \{x_N \le K |x'| + h_2\} \cap \Sigma_\beta,
\]
and $\{x_N \le K |x'| + h_2\} \cap \Sigma_\beta$ is bounded by construction.
\smallskip

We consider a point \begin{equation}\label{heart}
{\mbox{$x_0=(0',x_{0,N})$, with~$x_{0,N}\in (0,h_1)$.}}\end{equation}
Let $D$ be a bounded domain of $\R^N$ such that 
\[
\{x_N \le K |x'| + h_2\} \cap \Sigma_\beta \subset D \subset \Sigma_\beta,
\] 
with $\pa D \setminus \{0\}$ sufficiently smooth, see Figure~\ref{12345678}.

\begin{figure}
    \centering
    \includegraphics[width=11.8cm]{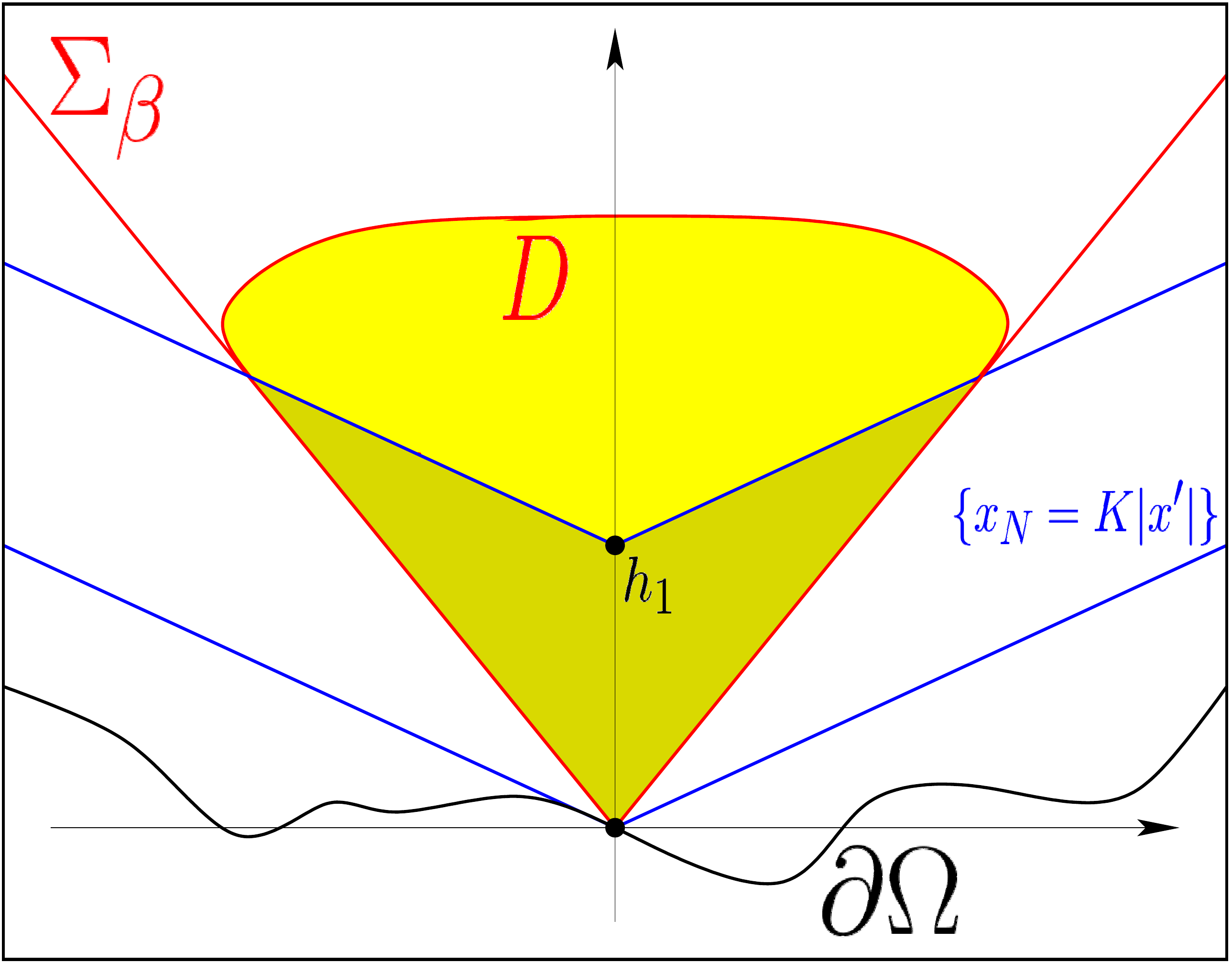}
    \caption{The set~$D$ in the proof of Theorem~\ref{thm: main 1}-($iii$).}
    \label{12345678}
\end{figure}

We introduce a continuous function $v$ satisfying
\begin{equation}\label{piupiuUJaaa:2}
\begin{cases}
(-\Delta)^s v = 0 & \text{in $D$} ,\\
v = 0 & \text{in $(\R^N \setminus D)\cap  \{x_N < K |x'| + h_1 \}$}, \\
v = \eps_1 & \text{in $(\R^N \setminus D)\cap  \{x_N > K |x'| + h_2\}$} ,\\
0 \le v \le \eps_1 & \text{in $\R^N \setminus D$}.
\end{cases}
\end{equation}
By the maximum principle, $v>0$ in $D$, and hence by compactness there exists $\bar C >0$ such that 
\begin{equation}\label{def bar C}
\text{$v (x)\ge \bar C$
for any~$x$
such that~$B_{h_1\sin\beta}(x)\subseteq{D}$
and~$x_N\ge  h_1$}.
\end{equation}
Also, by~\eqref{5.999}, we have that~$u \ge v$ in $\R^N \setminus D$.
Furthermore, $(-\Delta)^s u \ge 0$ in $\Omega \supset D$, 
thanks to assumption~($f1$) and
Theorem~\ref{thm: main 1}-($i$). Therefore,
by the
comparison principle in
Theorem \ref{thm: super max}, we find that 
\begin{equation}\label{OAJLMhJAaOIA1}
{\mbox{$v \le u$ in $\R^N$.
}}\end{equation}
Recalling~\eqref{heart}, let now~$r_0:= |x_0| \sin \beta = x_{0,N} \sin \beta$, and let us define
\[
r_{k+1}:= r_k(1+\tan \beta) \quad \text{and} \quad x_{k+1} = x_k + \left(0', \frac{r_k}{\cos \beta} \right) = \left(0', x_{k,N} + \frac{r_k}{\cos \beta} \right),
\] 
see Figure~\ref{FIGI2}.

\begin{figure}
    \centering
    \includegraphics[width=11.8cm]{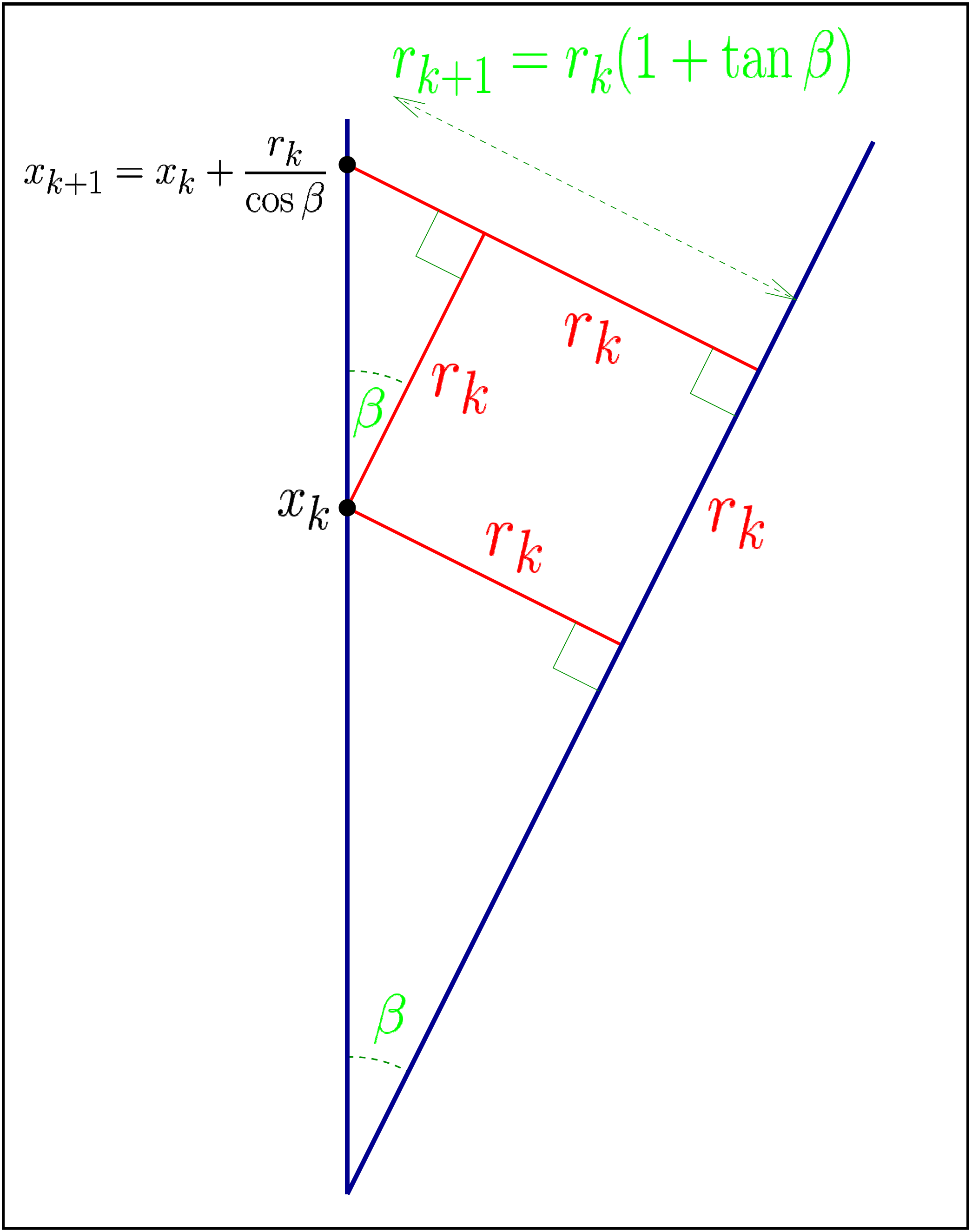}
    \caption{The geometry involved in the proof of Theorem~\ref{thm: main 1}-($iii$).}
    \label{FIGI2}
\end{figure}

In this way $B_{r_k}(x_k) \subset \Sigma_\beta$ for every $k$, and
\begin{equation}\label{radii}
r_k=r_0(1+\tan\beta)^k = x_{0,N} \,\sin \beta( 1+\tan\beta)^k.
\end{equation}
As a consequence
\begin{equation}\label{coordinates}
\begin{split}
x_{k,N} &= x_{0,N} + \sum_{i=0}^{k-1} (x_{i+1,N}-x_{i,N}) = x_{0,N}+ 
\sum_{i=0}^{k-1} x_{0,N} \,\tan \beta\,(1+\tan \beta)^i \\
&= x_{0,N}+ x_{0,N} \,\tan \beta\, \frac{(1+\tan \beta)^k-1}{\tan \beta} =  x_{0,N} \,
(1+\tan \beta)^k,
\end{split}
\end{equation}
for every $k \in \N$. 

We claim that 
\begin{equation}\label{croce}
\begin{split}
&{\mbox{there exists a first natural number $k_0$ depending on $x_0$ such that }}\\
&
x_{k_0} \in \{K |x'| + h_1<x_N\} \quad \text{and} \quad B_{r_{k_0}}(x_{k_0}) \subset \Sigma_\beta \cap B_{h_2} \subset D.
\end{split}\end{equation}
Thanks to \eqref{radii} and \eqref{coordinates}, and observing that by construction $x_{k}'=0'$ for every $k$, we see that $k_0 \in \N$ has to be the minimum natural number satisfying
\begin{equation}\label{iknHAHJAla1a}
x_{0,N}(1+\tan \beta)^{k_0} > h_1 \quad \text{and} \quad x_{0,N} (1+\sin \beta) ( 1+\tan\beta)^{k_0} <h_2, 
\end{equation}
that is
\begin{equation}\label{k_0}
\frac{\log \left(\frac{h_1}{x_{0,N}}\right) }{ \log \left(1+ \tan \beta\right)} \le k_0 \le  \frac{\log \left(\frac{h_2}{(1+\sin \beta)x_{0,N}}\right) }{ \log \left(1+ \tan \beta\right)}.
\end{equation}
Notice that such a $k_0 \in \N$ does exist, thanks to \eqref{additional}. This proves~\eqref{croce}.

Furthermore, we remark that 
\[
\frac{\log \left(\frac{h_1}{x_{0,N}}\right) }{ \log \left(1+ \tan \beta\right)} = \mu_1 - \nu_1 \log(x_{0,N}) \quad \text{and} \quad \frac{\log \left(\frac{h_2}{(1+\sin \beta)x_{0,N}}\right) }{ \log \left(1+ \tan \beta\right)} = \mu_2 - \nu_2 \log(x_{0,N})
\]
for suitable $\mu_1,\mu_2 \in \R$, and $\nu_1,\nu_2>0$, all depending only on $h_1$, $h_2$ and~$\beta$. 
This and~\eqref{k_0} say that
\begin{equation}\label{3.8bis}
\mu_1-\nu_1\,\log x_{0,N}\le k_0 \le\mu_2-\nu_2\,\log x_{0,N}.
\end{equation}
The previous construction implies that we have finite sequences of points $x_k$ and
radii $r_k$, with $k=0,\dots,k_0$, such that 
\begin{equation}\label{piupiuUJaaa}
B_{r_k}(x_k) \subset D\end{equation}
for every~$k=0,\dots,k_0$, and
\[
|x_{k+1}-x_{k}| = x_{k+1,N}-x_{k,N} = \frac{r_k}{\cos \beta} = \frac{r_{k+1}}{\cos \beta(1+\tan \beta)} =: \rho_\beta r_{k+1},
\] 
with $\rho_\beta <1$ independent of $k$. 

{F}rom~\eqref{piupiuUJaaa:2} and~\eqref{piupiuUJaaa}, we deduce that~$v$ 
is non-negative in $\R^N$ and $s$-harmonic in $B_{r_{k+1}}(x_{k+1})$, for any~$k=0,\dots,k_0-1$.
Hence,
we are in position to apply the Harnack inequality in~\cite[Theorem 5.1]{CaffSilvCPDE}, 
deducing that 
\[
v(x_{k+1}) \le \sup_{B_{\rho_\beta r_{k+1}}(x_{k+1})} v \le \tilde C  
\inf_{B_{\rho_\beta r_{k+1}}(x_{k+1})} v  \le \tilde C v(x_k)
\]
for every $k=0,\dots, k_0-1$, where $\tilde C>0$ is a positive constant depending 
only on $\beta$. 

Now we observe that
$$ r_{k_0} = x_{0,N} \,\sin \beta( 1+\tan\beta)^{k_0} \ge h_1\sin\beta,$$
thanks to~\eqref{radii} and~\eqref{iknHAHJAla1a}. Similarly,
$$ x_{k_0,N} =  x_{0,N} \,
(1+\tan \beta)^{k_0}\ge h_1,$$
thanks to~\eqref{coordinates} and~\eqref{iknHAHJAla1a}.

As a consequence of this and of~\eqref{def bar C}, we obtain that~$v(x_{k_0})\ge\bar C$.
Therefore, using \eqref{OAJLMhJAaOIA1} and~\eqref{3.8bis}, we obtain 
\begin{align*}
u(x_0) & \ge v(x_0)  \ge \frac{1}{\tilde C^{k_0}} v(x_{k_0}) \ge \frac{\bar C}{e^{k_0 \log \tilde C}}
\\
&\ge \frac{\bar C}{e^{\mu_2 \log \tilde C} \cdot e^{-\nu_2 \log \tilde C \log x_{0,N}}} = C x_{0,N}^{\bar \rho} 
= C (x_{0,N}-\varphi(x_0'))^{\bar \rho},
\end{align*}
for some positive constants $C,\bar \rho$ depending only on $\beta$, $K$, $h_1$ and $h_2$. 

This gives the desired point-wise estimate at the point $x_0$, with~$x_0'=0$ and~$x_{0,N}\in(0,h_1)$
(recall~\eqref{heart}),
and with constants~$C$ 
and $\bar \rho$ depending only on $\beta$, $K$, $h_1$ and $h_2$. Up to a translation,
the same estimate holds at any point of~$\Omega$, with vertical coordinate~$x_N
\in \big(\varphi(x'),\,\varphi(x')+h_1\big)$.
That is,
translating $\Sigma_\beta$ (and hence $D$) along the
boundary $\pa \Omega$, the family $\{v(x-x_0): x_0 \in \pa \Omega\}$ 
gives a wall of lower barriers for $u$, providing the 
desired lower estimate at any point of $\{\varphi(x') < x_N < \varphi(x')+h_1\}$,
which completes the proof of
Theorem~\ref{thm: main 1}-($iii$).\end{proof}

\subsection{Uniqueness and monotonicity of the positive solution}\label{sec: uniqueness} 

This section is devoted to the proof of point ($v$) in Theorem \ref{thm: main 1}. Our approach is inspired by \cite[Section 5]{BCNCPAM}, but we modify it in such a way that the proof gives essentially in one shot both uniqueness and monotonicity of the solution. This simplifies the argument in \cite{BCNCPAM}, and completes the proof of Theorem \ref{thm: main 1}. 

We start with the preliminary observation that, under our assumptions, solutions are bounded away from $1$ if the distance from the boundary is finite (recall~\eqref{mu=1}).

\begin{lemma}\label{lem: 3.4 BCN}
For $h >0$, let 
$$\Omega_h:= \{ \varphi(x') < x_N < \varphi(x')+h\}.$$
Then, any bounded solution to \eqref{problem} is bounded away from $1$ in $\Omega_h$, namely
$$ \sup_{\overline{\Omega_h}} u < 1.$$
\end{lemma}

\begin{proof}
Let $u$ be a solution
of~\eqref{problem}, and let us suppose by contradiction that $u(x_n) \to 1$ 
as $n\to+\infty$
along a sequence $\{x_n\} \subset \Omega_h$. 

Let us consider then the sequence of translated functions $u_n(x):= u(x+x_n)$.
We observe that
\begin{equation}\label{87654erfgIGHaaJa}
\lim_{n\to+\infty} u_n(0)=1\end{equation}
and
\begin{equation}\label{86543rtGHJFDFGh}
\begin{cases}
(-\Delta)^s u_n = f(u_n) & \text{in $\Omega_n$}, \\
u_n = 0 & \text{in $\R^N \setminus \Omega_n$}, \\
0 \le u_n \le 1 & \text{in $\R^N$},
\end{cases} 
\end{equation}
where
$$ \Omega_n:= \{x_N > \varphi(x'+x'_n) - x_{n,N}\}.$$
Since $\varphi$ is globally Lipschitz continuous, the translated epigraphs $\Omega_n$ converge, up to subsequence, to a limit epigraph $\bar \Omega$, which is
still globally Lipschitz continuous. Notice also that $\Omega_n$ satisfies a uniform cone condition, with opening of the cone independent of $n$. Thus, by Theorem \ref{thm: boundary regularity}, the sequence $\{u_n\}$ is uniformly bounded in $\mathcal{C}^{0,\alpha}(\R^N)$, and hence converges locally uniformly in $\R^N$ to a limit function~$\bar u \in \mathcal{C}^{0,\alpha}(\R^N)$. 

Thus, by~\eqref{86543rtGHJFDFGh}
and the stability property of viscosity solutions \cite[Lemma 4.5]{CafSilCPAM}, we infer that $\bar u$ is a viscosity solution to
\begin{equation}\label{789JA654588888}
\begin{cases}
(-\Delta)^s \bar u = f(\bar u) & \text{in $\bar \Omega$}, \\
\bar u = 0 & \text{in $\R^N \setminus \bar \Omega$}, \\
0 \le \bar u \le 1 & \text{in $\R^N$}.
\end{cases}
\end{equation}
Actually, arguing as in \cite[Remark 2.3]{QuaXia}, we see that~$\bar u$ is a classical solution.
Since also the function constantly equal to~$1$ is a solution
of the equation in~\eqref{789JA654588888}
(recall that~$f(1)=0$), 
the strong maximum principle implies that $\bar u<1$ in $\R^N$. 

This is in contradiction with the fact that~$\bar u(0)=1$,
which follows from~\eqref{87654erfgIGHaaJa}.
\end{proof}

In order to prove Theorem \ref{thm: main 1}-($v$), let us consider two bounded solutions $u$ and $v$ of \eqref{problem}. We show that necessarily $u \ge v$. Exchanging the role of $u$ and $v$ we deduce also that $v \ge u$, whence uniqueness follows. 

For this, first of all we observe that,
by Theorem \ref{thm: main 1}-($ii$), there exists $A>0$ such that 
\begin{equation}\label{9ioj56789oiUYTRFGHAAAAAA}
{\mbox{$u(x),v(x) > {t_{1}}$ if $\dist(x,\pa \Omega) > A$,}}\end{equation}
where ${t_{1}}$ was introduced in assumption ($f3$).
Let 
\[
\Omega_A := \left\{ x \in \Omega: \dist(x,\pa \Omega) < A\right\} \quad{\mbox{  and }}\quad \Omega^A:= \Omega \setminus \overline{\Omega_A}.
\]
Also, for $\tau \ge 0$ let us consider
\[
u_\tau(x):= u(x+\tau e_N).
\]
As in \cite[Lemma 5.1]{BCNCPAM}, we show that:
\begin{lemma}\label{lem: 5.1 BCN}
If $u_\tau > v$ in $\Omega_A$, then $u_\tau > v$ in~$\R^N$.
\end{lemma}
\begin{proof}
First of all, since $\Omega$ is an epigraph we have
\begin{equation}\label{6g8uj9ikmuhjn}
{\mbox{if~$x\in\Omega^A$, then~$x+\tau e_N\in \Omega^A$.}}
\end{equation}

Now we notice that~$u_\tau \ge0= v$ in $\R^N \setminus \Omega$.
Thus, to establish the 
desired result,
we have only to prove that \begin{equation}\label{87uygf8ytfdijuhvijuhv}
{\mbox{$u_\tau \ge v$ in $\Omega^A$.}}\end{equation} To this aim, we 
use~\eqref{6g8uj9ikmuhjn} and we
observe that 
\[
\begin{cases}
(-\Delta)^s (v-u_\tau) - c_\tau(x) (v-u_\tau) =0 \quad \text{in $\Omega^A$} ,\\
v - u_\tau \le 0 \quad \text{in $\R^N \setminus \Omega^A$},
\end{cases}
\]
where 
\[
 c_\tau(x):= \begin{cases} \displaystyle\frac{f(v(x))-f(u_\tau(x))}{
v(x)-u_\tau(x)} & \text{if $v(x) \neq u_\tau(x)$}, \\
 0 & \text{if $v(x) = u_\tau(x)$}. \end{cases}
\]
Now we claim that
\begin{equation}\label{6g8uj9ikmuhjn:3}
{\mbox{$u_\tau,v \ge t_1$ in $\Omega^A$.}}\end{equation}
Indeed, let~$x\in \Omega^A$. Then~$x+\tau e_N\in\Omega^A$, thanks to~\eqref{6g8uj9ikmuhjn}.
Hence~\eqref{6g8uj9ikmuhjn:3}
follows from~\eqref{9ioj56789oiUYTRFGHAAAAAA}.

Now, since $f$ is non-increasing in $(t_1,1)$,
we deduce from~\eqref{6g8uj9ikmuhjn:3}
that $c_\tau \le 0$ in $\Omega^A$, and by Lipschitz continuity we have also $c_\tau \in L^\infty(\Omega^A)$.

Moreover, it is clear that $\Omega^A$ satisfies an exterior cone condition (since $\Omega$ does), and hence 
Theorem \ref{thm: super max} and the strong maximum principle imply~\eqref{87uygf8ytfdijuhvijuhv},
as desired.\end{proof}

Now, we aim at showing that $u_\tau \ge v$ in $\R^N$ for $\tau =0$. 
Thanks to Lemma \ref{lem: 5.1 BCN}, this statement is equivalent to showing that~$u_\tau \ge v$ in $\Omega_A$ for $\tau = 0$.

By Lemma \ref{lem: 3.4 BCN}, we know that $v \le C$ with $C<1$ in $\Omega_A$.
Moreover, since $u \to 1$ uniformly as $\dist(x,\pa \Omega) \to +\infty$ (recall 
Theorem \ref{thm: main 1}-($ii$)), we have 
that~$u_\tau \ge v$ in $\Omega_A$ for $\tau$ sufficiently large. Therefore, 
we can define
\begin{equation}\label{def T}
T:= \inf\left\{ \tau >0: u_t > v \text{ in $\Omega_A$ for every $t > \tau$ }\right\}\in[0,+\infty).
\end{equation}

\begin{remark}\label{rem: uniq and monot}
For the uniqueness, one could replace the previous definition of $T$  with
\[
\inf\{\tau >0: u_\tau \ge v \text{ in $\Omega_A$}\},
\]
as done in \cite{BCNCPAM}. Nevertheless, as we will show later, definition \eqref{def T} permits to perform
the same argument used for the uniqueness also for the monotonicity of $u$.
\end{remark}

We are now in the position of completing the proof of Theorem \ref{thm: main 1}-($v$).

\begin{proof}[Completion of the proof of Theorem \ref{thm: main 1}-($v$)]
By continuity $u_T \ge v$ in $\Omega_A$. Hence, by
Lemma \ref{lem: 5.1 BCN},
\begin{equation}\label{hg1qaxUYFDX}
{\mbox{$u_T \ge v$ in $\R^N$.}}
\end{equation}
Thus, if $T=0$ the proof is complete. To rule out the possibility that~$T>0$, we argue by contradiction.

If $T>0$, then there exist sequences $0<\tau_j<T$ and $x_j \in \Omega_A$ such that
\begin{equation}\label{absurd T:0}
\lim_{j\to+\infty}\tau_j=T>0
\end{equation}
and
\begin{equation}\label{absurd T}
u(x_j+ \tau_j e_N) \le v(x_j). 
\end{equation}
Let us consider 
\[
u_j(x):= u(x+x_j) \quad \text{and} \quad v_j(x) := v(x+x_j).
\]
As in Lemma \ref{lem: 3.4 BCN}, we have
that, up to subsequences, $u_j \to \bar u$ and $v_j \to \bar v$ locally uniformly, and $\bar u$ and $\bar v$ are 
solutions to \eqref{problem} in a limit epigraph $\bar \Omega$. 

We remark that, since~$x_j\in\Omega$, the point~$0$ belongs to the approximating domains
and therefore
\begin{equation}\label{789-A7890AA}
{\mbox{$0$ belongs to the closure of~$\bar\Omega$.}}
\end{equation}
We also notice that
\begin{equation}\label{hg1qaxUYFDX:2}
{\mbox{$\bar u_T (x):=\bar u(x+Te_N)\ge \bar v(x)$ for any $x\in\R^N$,}}
\end{equation}
thanks to~\eqref{hg1qaxUYFDX}. Furthermore,
in light of \eqref{absurd T:0},
\eqref{absurd T} and using the uniform convergence,
\begin{equation}\label{claim 5.7}
\begin{split}&\bar u_T(0) =
\bar u(T e_N) =\lim_{j\to+\infty} u_j (\tau_j e_N)
=\lim_{j\to+\infty} u(x_j+\tau_j e_N)
\\ &\qquad\le \lim_{j\to+\infty} v(x_j)
=\lim_{j\to+\infty} v_j(0)
= \bar v(0).\end{split}
\end{equation}
By~\eqref{hg1qaxUYFDX:2}
and~\eqref{claim 5.7}, we conclude that
\begin{equation}\label{8yugfcijughvoijuhbgiuhgfghjytyuiuhgfGHJ}
\bar u_T(0)=\bar v(0).
\end{equation}
Now we claim that
\begin{equation} \label{9iuhgviuhgiuHIJHBAAA:XX}
{\mbox{$\bar u_T >0 = \bar v$ on $\pa \bar \Omega$.}}\end{equation}
To check this, we take~$\bar p\in\pa \bar \Omega$.
Then there exists~$x_0\in\pa\Omega$ such that~$p_j:=x_0-x_j\to\bar p$
as $j\to+\infty$. As a consequence, using the uniform convergence, we see that
\begin{equation}\label{9iuhgviuhgiuHIJHBAAA:1}
\bar v(p) = \lim_{j\to+\infty} v_j(p)
= \lim_{j\to+\infty} v_j(p_j) = \lim_{j\to+\infty} v(p_j+x_j) = v(x_0)=0
\end{equation}
and
\begin{equation}\label{9iuhgviuhgiuHIJHBAAA:2}\begin{split}
&\bar u_T(p) = \bar u(p+Te_N)=
\lim_{j\to+\infty} u_j(p+Te_N)= \lim_{j\to+\infty} u_j(p_j+Te_N)\\&\qquad = \lim_{j\to+\infty} 
u(p_j+x_j+Te_N) =u(x_0+Te_N)>0,\end{split}
\end{equation}
since~$x_0+Te_N\in\Omega$ (here we are using that~$T>0$).
Combining~\eqref{9iuhgviuhgiuHIJHBAAA:1}
and~\eqref{9iuhgviuhgiuHIJHBAAA:2}, we obtain~\eqref{9iuhgviuhgiuHIJHBAAA:XX}.

Also,
\[
(-\Delta)^s ( \bar u_T-\bar v) -c_T(x) (\bar u_T-\bar v) = 0 \qquad \text{in $\bar \Omega$},
\]
where
\[
c_T(x) := \begin{cases}  \displaystyle\frac{f(\bar u_T(x))-f(\bar v(x))}{\bar u_T(x)-\bar v(x)}, & \text{if $\bar u_T(x) \neq \bar v(x)$} \\ 0 & \text{if $\bar u_T(x) = \bar v(x)$},
\end{cases}
\]
and $c_T \in L^\infty(\bar \Omega)$. Thus, the strong maximum principle and~\eqref{hg1qaxUYFDX:2}
imply that either $\bar u_T > \bar v$ in $\bar \Omega$, or $\bar u_T \equiv \bar v$ in $\R^N$. 

But the latter alternative is not admissible, due to~\eqref{9iuhgviuhgiuHIJHBAAA:XX}. 
Therefore, we conclude that~$\bar u_T > \bar v$ in $\bar \Omega$ and so,
again by \eqref{9iuhgviuhgiuHIJHBAAA:XX}, in the closure of~$\bar \Omega$. This is in contradiction with~\eqref{789-A7890AA}
and~\eqref{8yugfcijughvoijuhbgiuhgfghjytyuiuhgfGHJ}, hence the proof
of Theorem \ref{thm: main 1}-($v$) is complete.
\end{proof}

\begin{proof}[Proof of Theorem \ref{thm: main 1}-($vi$)]
We proceed exactly as for the uniqueness, but instead of comparing $u_{\tau}$ with $v$ we compare $u_{\tau}$ with $u$ (indeed, $v$ was just the generic solution,
so the case~$v:=u$ is admissible). In the end, we obtain that~$u_{\tau}>u$ in $\Omega_A$ for every $\tau>0$, which by Lemma \ref{lem: 5.1 BCN} yields the desired monotonicity.
\end{proof}

\section{Monotonicity of solutions in coercive epigraphs}\label{8uhg6tr4rtgbifdfghjAHJ}

This section is devoted to the proof of Theorem \ref{thm: main 2}, which rests upon the moving planes method. We introduce some notation: for $\lambda \in \R$, we set
\[
\begin{split}
T_\lambda & := \{x \in \R^N: x_N=\lambda\}; \\
H_{\lambda} & := \{x \in \R^N: x_N <\lambda\}; \\
x^\lambda & := (x',2\lambda-x_N) \quad \text{the reflection of $x$ with respect to $T_\lambda$}; \\
A^\lambda&:= \text{the reflection of a given set $A$ with respect to $T_\lambda$}; \\
\Sigma_\lambda &:= H_\lambda \cap \Omega;\\
\lambda_0 &:= \inf \left\{x_N: \text{there exists $x' \in \R^N$ with $(x',x_N) \in \Omega$}\right\}.
\end{split}
\]
The crucial remark is that, since we deal with a coercive epigraph, the set $\Sigma_\lambda$ is bounded for every $\lambda \in \R$, even if $\Omega$ is unbounded. Therefore, one can adapt the proof of \cite[Theorem 1.1]{FeWa}, which uses the moving planes method for fractional elliptic equations in bounded domains. For the reader's convenience, we recall the following weak maximum principle in sets of small measure, which we conveniently re-phrase for our purpose.

\begin{proposition}[Proposition 2.2, \cite{FeWa}]\label{lem: small}
Let $D$ be an open and bounded subset of $\R^N$. Let $c \in L^\infty(D)$ with $\|c\|_{L^\infty(D)} < M$, and let $z$ be a solution to 
\begin{equation}\label{comparison small}
\begin{cases}
(-\Delta)^s z \ge c(x) z & \text{in $D$}, \\
z \ge 0 & \text{in $\R^N \setminus D$}.
\end{cases}
\end{equation}
Then, there exists $\delta>0$ depending only on $N$, $s$ and $M$ such that if $|D|<\delta$
then~$z \ge 0$ in $D$.
\end{proposition}

\begin{remark}\label{rem: uniform delta}
Consider a sequence of boundary value problems of type \eqref{comparison small}, with $c= c_n$ and $D=D_n$, $n \in \N$. If we have a uniform bound $\|c_n\|_{L^\infty(D_n)} < M$, then 
Proposition~\ref{lem: small}
gives a threshold $\delta$ independent of $n$. 
\end{remark}

\begin{proof}[Proof of Theorem \ref{thm: main 2}]
We set~$u_\lambda(x):=u(x',2\lambda-x_N)$ and~$w_\lambda(x):=u_\lambda(x)-u(x)$.

We aim at proving that $w_\lambda > 0$ in $H_\lambda$ for every $\lambda>\lambda_0$, which gives the desired monotonicity. 

For any~$\lambda>\lambda_0$, we have that~$2\lambda-x_N > x_N$ in $\Sigma_\lambda$.
Accordingly, the monotonicity of $f$ in $x_N$ gives
\begin{equation}\label{pb refl}
\begin{split}
(-\Delta)^s w_\lambda (x) & = (-\Delta)^s (u_\lambda(x)-u(x)) \\
& = f(x', 2\lambda-x_N, u_\lambda(x)) - f(x,u(x)) \\
& \ge f(x, u_\lambda(x)) - f(x,u(x)) \\
& = c_\lambda(x) w_\lambda(x) 
\end{split}
\end{equation}
in $\Sigma_\lambda$, with 
\[
c_\lambda(x):= \begin{cases}
\displaystyle\frac{f(x,u_\lambda(x))-f(x,u(x))}{u_\lambda(x)-u(x)} & \text{if $u_\lambda(x) \neq u(x)$}, \\
0 & \text{if $u_\lambda(x) = u(x)$}.
\end{cases}
\]
Notice that,
thanks to the Lipschitz continuity of $f$ and the fact that $u \in L^\infty_\loc(\R^N)$, for any $\bar \lambda>\lambda_0$,
there exists $C>0$ such that $\|c_\lambda\|_{L^\infty(\Sigma_\lambda)} \le C$ for any $\lambda \in (\lambda_0,\bar \lambda]$.

For convenience, we now divide the proof into separate steps:\medskip

\emph{Step 1) We show that 
$w_\lambda > 0$ in $\Sigma_\lambda$ for any~$\lambda>\lambda_0$, with~$\lambda-\lambda_0$ 
small enough.} 

For this, let $\Sigma_\lambda^-:= \{x \in \Sigma_\lambda: w_\lambda(x) <0\}$. 
We first show that
\begin{equation}\label{8ikjYHIJAA:A}
{\mbox{$w_\lambda \ge 0$ in $\Sigma_\lambda$ for any~$\lambda>\lambda_0$, with~$\lambda-\lambda_0$ 
small enough,}}
\end{equation}
i.e., that~$\Sigma_\lambda^- =\varnothing$.
To this aim, we argue by contradiction.
If $\Sigma_\lambda^- \neq 
\varnothing$, we can define
\begin{equation}\label{def w_1 e w_2}
w_{1,\lambda}:= \begin{cases} w_\lambda & \text{in $\Sigma_\lambda^-$}  \\ 0, 
& \text{in $\R^N \setminus \Sigma_\lambda^-$},    \end{cases}  \qquad{\mbox{ and }}\qquad 
w_{2,\lambda}:= \begin{cases} w_\lambda & \text{in $\R^N \setminus \Sigma_\lambda^-$},\\ 0 & \text{in $\Sigma_\lambda^-$}.    \end{cases}
\end{equation}
We observe that~$w_\lambda = w_{1,\lambda} + w_{2,\lambda}$, and that $w_{1,\lambda} \le 0$ while $w_{2,\lambda} \ge 0$. Exactly as in \cite[Theorem 1.1, step 1]{FeWa}, it is possible to show that $(-\Delta)^s w_{2,\lambda} \le 0$ in $\Sigma_\lambda^-$. Hence, by \eqref{pb refl},
\[ 
\begin{cases}
(-\Delta)^s w_{1,\lambda} \ge c_\lambda(x) w_{1,\lambda} & \text{in $\Sigma_\lambda^-$}, \\
w_{1,\lambda} = 0 & \text{in $\R^N \setminus \Sigma_\lambda^-$}.
\end{cases}
\]
Thus, by the maximum principle in sets of small measure
(see Proposition~\ref{lem: small}), we infer that, for 
any~$\lambda > \lambda_0$ close to $\lambda_0$,
it results that~$w_\lambda \ge 0$ in~$\R^N$. 
As a consequence, $w_{1,\lambda}=w_\lambda\ge0$ in~$\Sigma^-_\lambda$,
which proves~\eqref{8ikjYHIJAA:A}.

As a side remark, we
notice that here we do not need $u \in L^\infty(\R^N)$, but only $u \in L^\infty_{\loc}(\R^N)$ (which follows automatically by the
definition of classical or even viscosity solution), since $\Sigma_\lambda$ is bounded.

Now we claim that 
\begin{equation}\label{claim on strong}
\text{if $\lambda > \lambda_0$ and $w_\lambda \ge 0$ in $\Sigma_\lambda$, then $w_\lambda>0$ in $\Sigma_\lambda$}.
\end{equation}
This is not a consequence of the strong maximum principle since the function $w_\lambda$ changes sign, by definition.

By contradiction, let us suppose that there exists $x_0 \in \Sigma_\lambda$ such that $u_\lambda(x_0) = u(x_0)$. Since $|x_0-y| < |x_0-y_\lambda|$ for every $x_0 \in \Sigma_\lambda$ and $y \in H_\lambda$, and~$w_\lambda$ is positive in a subset of $H_\lambda$ having positive measure, we deduce that 
\begin{align*}
0 &=c_\lambda(x_0)\,w_\lambda(x_0)\le 
(-\Delta)^s w_\lambda(x_0) = -\int_{H_\lambda} \frac{w_\lambda(y)}{|x-y|^{N+2s}}\,dy -\int_{\R^N \setminus H_\lambda} \frac{w_\lambda(y)}{|x-y|^{N+2s}}\,dy \\
& = - \int_{H_\lambda}  w_\lambda(y) \left( \frac{1}{|x-y|^{N+2s}} - \frac{1}{|x-y_\lambda|^{N+2s}} \right)\,dy < 0,
\end{align*}
where~\eqref{pb refl} was used, and so we obtain
a contradiction. 

This proves~\eqref{claim on strong}.
Then, the desired result in Step~1 follows by combining~\eqref{8ikjYHIJAA:A}
and~\eqref{claim on strong}.\medskip

\emph{Step 2) We show that $w_\lambda>0$ in $\Sigma_\lambda$ for every $\lambda > \lambda_0$.} Let 
\[
\tilde \lambda:= \sup \{\lambda >\lambda_0: \text{$w_\mu>0$ in $\Sigma_\mu$ for every $\mu \in (\lambda_0, \tilde \lambda)$}\}.
\]
By the previous step $\tilde \lambda>\lambda_0$. If $\tilde \lambda = +\infty$ the proof of
Step~2 is complete, and hence we argue by contradiction supposing that $\tilde \lambda < +\infty$. 

By continuity and by \eqref{claim on strong} we have 
\begin{equation}\label{9ijnyhgbRTAYUJUAUA78AA}
{\mbox{$w_{\tilde \lambda}>0$ in $\Sigma_{\tilde \lambda}$.}}\end{equation} Let us consider now
\[
m:= \sup_{\lambda \in (\lambda_0, \tilde \lambda +1]} \|c_\lambda\|_{L^\infty(\Sigma_\lambda)}.
\]
This value is finite since $u \in L^\infty_\loc(\R^N)$, $f$ is locally Lipschitz, and $\Sigma_{\tilde \lambda+1}$ is bounded. Therefore, 
the threshold $\delta = \delta(N,s, m)$ for the maximum principle in domains of small 
measure 
in Proposition~\ref{lem: small}
is well defined
(recall Remark \ref{rem: uniform delta}). 

Let us fix a compact set $K \Subset 
\Sigma_{\tilde \lambda}$ such that
\begin{equation}\label{ijjmikjmokFTYUIA}
|\Sigma_{\tilde \lambda} \setminus K|< \frac\delta2.\end{equation}
By 
compactness and~\eqref{9ijnyhgbRTAYUJUAUA78AA},
we have that
$$\inf_K w_{\tilde \lambda} >0.$$ 
Using this and~\eqref{ijjmikjmokFTYUIA}, we have that, by continuity,
there exists $\bar 
\eps>0$ small enough such that
\begin{equation}\label{ijjmikjmokFTYUIA:XXXX}
|\Sigma_{\tilde \lambda+\eps} \setminus K|< \delta
\qquad{\mbox{ and }}\qquad\inf_K w_{\tilde \lambda + \eps} > 0\end{equation} for every $\eps \in (0,\bar \eps)$. 

Let now
$\Sigma_{\tilde \lambda +\eps}^-:= 
\Sigma_{\tilde \lambda +\eps}
\cap
\{ w_{\tilde \lambda +\eps} <0 \}$.
We observe that 
\begin{equation}\label{9iugh4rTGyhja9uygfcUJ}
{\mbox{the measure of~$\Sigma_{\tilde \lambda +\eps}^-$
is smaller than $\delta$, 
}}\end{equation}
thanks to~\eqref{ijjmikjmokFTYUIA:XXXX}.

Now, we consider the functions~$w_{1,\tilde \lambda + \eps}$ and 
$w_{2,\tilde \lambda + \eps}$ defined as in \eqref{def w_1 e w_2} with~$\lambda:=
\tilde \lambda + \eps$. Proceeding as in Step 1, 
we can check that
\[
(-\Delta)^s w_{1,\tilde \lambda + \eps} \ge c_{\tilde \lambda+\eps}(x) w_{1,\tilde \lambda+\eps} \qquad \text{in $\Sigma_{\tilde \lambda +\eps}^-$}.
\]
We use this, \eqref{9iugh4rTGyhja9uygfcUJ}
and the maximum principle in sets of small measure 
(see Proposition~\ref{lem: small}) to conclude that~$w_{1,\tilde \lambda+\eps} \ge 0$ in $\R^N$ for every $\eps$ sufficiently small.

As a consequence, recalling~\eqref{def w_1 e w_2}, we have that~$w_{\tilde \lambda + \eps}\ge
w_{1,\tilde \lambda + \eps}\ge0$. Therefore, 
from~\eqref{claim on strong} we conclude that~$w_{\tilde \lambda + \eps} >0$ in $\Sigma_{\tilde \lambda + \eps}$ for any $\eps>0$ small enough, in contradiction with the definition of $\tilde \lambda$.
\end{proof}

\section{Overdetermined problems}\label{sec: overdet}

This section concerns the study of the overdetermined problem \eqref{overdet},
where $\Omega$ is a globally Lipschitz epigraph, satisfying the additional flatness assumption \eqref{hp epigrafico}. Regarding $f$, it satisfies ($f1$)-($f3$) in the introduction. As in the proof of Theorem \ref{thm: main 1}, it is not not restrictive to suppose that $\mu=1$.

In particular, we now proceed with the proof of Theorem \ref{thm: overdet}, which is
the fractional counterpart of the proof of Theorem~7.1 in~\cite{BCNCPAM},
where the local case was considered. 

\begin{proof}[Proof of Theorem \ref{thm: overdet}]
We claim that for any~$\tau' \in \R^{N-1}$
\begin{equation}\label{th overdet}
\Omega\subseteq\Omega - (\tau',0)=
\{ x: (x'+\tau',x_N) \in \Omega\} .
\end{equation}
To this extent,
for $\tau = (\tau',0)$ fixed and $h \ge 0$, let us consider
\[
\Sigma_{h,\tau}:= \Omega- \tau - h e_N = \{ x\in\R^N\,:\; (x'+\tau', x_N +h ) \in \Omega\}.
\]
Since $\varphi$ is Lipschitz continuous, for $h>0$ sufficiently large we have that~$\Sigma_{h,\tau}$ contains strictly $\Omega$. In other words, we can define the real number 
\begin{equation}\label{8ygcUADAFGYUIUYTFA}
h^*:= \inf\left\{ h \ge 0: \Sigma_{k,\tau} \supset \Omega \text{ for every $k >h$}\right\}
.\end{equation}
We claim that
\begin{equation}\label{98ikmA3456789098765AA}
h^* = 0.\end{equation}
To prove this,
let us suppose by contradiction that $h^*>0$. Then there exist sequences $0<h_j<h^*$ and $
x_j \in \Omega \setminus \Sigma_{h_j,\tau}$ with
$$ \lim_{j\to+\infty} h_j=h^*>0.$$
By assumption \eqref{hp epigrafico}, we infer that $\{x_j\}$ is bounded, and hence up to a subsequence $x_j \to a$, for some~$a\in\overline{\Omega \setminus \Sigma_{h^*,\tau}}\subseteq
\overline\Omega \setminus \Sigma_{h^*,\tau}$.

On the other hand, by~\eqref{8ygcUADAFGYUIUYTFA}, we know that~$ \Sigma_{h^*,\tau} \supseteq \Omega$
and therefore~$a \in \pa \Omega \cap \pa \Sigma_{h^*,\tau}$. 
In other words, the set $\Omega$ is internally tangent to $\Sigma_{h^*,\tau}$ in $a$.

Now, let us consider, for $h \ge h^*$,
\[
u_{h,\tau}(x):= u(x+\tau + he_N).
\]    
We claim that 
\begin{equation}\label{1141}
u_{h^*,\tau} \ge u \qquad \text{in $\R^N$}.
\end{equation} 
To this aim, we argue as in Subsection \ref{sec: uniqueness}. First, 
we introduce $A>0$ large enough, so that both $u$ and $u_{h^*,\tau}$ are larger than $t_1$ in $\Omega^A:= \{\dist(x,\pa \Omega) > A\}$ 
($t_1$ is defined in assumption ($f3$)).

Then, for any $h \ge h^*$, we have that~$u_{h,\tau} \ge t_1$ in $\Omega^A$, and $u_{h,\tau} \ge0= u$ in $\R^N \setminus \Omega$. 

By Lemma \ref{lem: 3.4 BCN} and Theorem \ref{thm: main 1} applied to $u$, we know that $u_{h,\tau} \ge u$ in $\Omega_A=\Omega\setminus\overline{\Omega^A}$ if $h$ is sufficiently large. Therefore we can define 
\[
\tilde h:= \inf\{h > h^*:  u_{k,\tau} \ge u \text{ in $\Omega_A$ for every $k > h$}\}.
\]
If $\tilde h= h^*$, then
claim \eqref{1141} follows
from Lemma \ref{lem: 5.1 BCN}.
On the other hand, if $\tilde h> h^*$ it is not difficult to obtain a contradiction as in Subsection \ref{sec: uniqueness},
thus completing the proof of~\eqref{1141}.

Moreover, by internal tangency, the outer normal to $\pa \Omega$ and to $\pa \Sigma_{h^*,\tau}$ coincide at the point~$a$.
Accordingly, by the $s$-Neumann condition in \eqref{overdet} reads
\begin{equation}\label{1142}
(\pa_\nu)_s u(a) - (\pa_\nu)_s u_{h^*,\tau}(a) = 0.
\end{equation}
On the other hand, the function $w_{h^*,\tau}:= u_{h^*,\tau}-u$ satisfies
\[
\begin{cases}
(-\Delta)^s w_{h^*,\tau} -c_{h^*,\tau}(x) w_{h^*,\tau} = 0 & \text{in $\Omega$}, \\
w_{h^*,\tau} \ge 0 & \text{in $\R^N$},
\end{cases}
\]
where 
\[
c_{h^*,\tau}(x):= \begin{cases} \displaystyle\frac{f(u_{h^*,\tau}(x))-f(u(x))}{
u_{h^*,\tau}(x)-u(x)} & \text{if $u_{h^*,\tau}(x)
\neq u(x)$}, \\
 0 & \text{if $u_{h^*,\tau}(x)=u(x)$}. \end{cases}
\]
Therefore, the Hopf lemma for the fractional Laplacian (see \cite[Lemma 1.2]{GreSer}) gives
\[
0> (\pa_\nu)_s w_{h^*,\tau}(a) = (\pa_\nu)_s u_{h^*,\tau}(a) - (\pa_\nu)_s u(a),
\]
in contradiction with \eqref{1142}.
This proves~\eqref{98ikmA3456789098765AA}.

Using~\eqref{98ikmA3456789098765AA}, we deduce that
\begin{equation*}
\Sigma_{0,\tau} = \Omega-\tau \supseteq \Omega,\end{equation*}
which in turn implies~\eqref{th overdet}.

Now, we deduce from~\eqref{th overdet}
that
the function $\varphi$ is necessarily a constant, i.e. $\Omega$ is a half-space
(and this concludes the proof of
Theorem \ref{thm: overdet}).

Indeed, if by contradiction $\varphi$ is not constant there exist $x_1',x_2' \in \R^{N-1}$
such that $\varphi(x_1')< \varphi(x_2')$. Let $y:= (\varphi(x_1') + \varphi(x_2'))/2$.
Notice that~$\varphi(x_1')<y<\varphi(x_2')$, and therefore 
\begin{equation}\label{ikjnmA78} 
(x_2',y) \not \in \Omega\quad{\mbox{ and }}\quad(x_1',y) \in \Omega.\end{equation}
Thus, using~\eqref{th overdet} with~$\tau':=x_2'-x_1'$, we obtain that
$$(x_1',y)=x_1  \in \Omega\subseteq \Omega -(x_2'-x_1',0).$$
By adding~$(x_2'-x_1',0)$ to this inclusion, we find that~$(x_2',y)\in\Omega$,
which is 
in contradiction with~\eqref{ikjnmA78}.
\end{proof}

\section*{Acknowledgements}

In a preliminary version of this paper (see~\cite{OLD}), the proof of Lemma~\ref{lem: frac Euler}
was unnecessarily complicated: we are indebted to Mouhamed
Moustapha Fall for the simpler argument that we incorporated
in the present version of this paper.\medskip

Part of this work was carried out while Serena Dipierro
and Enrico Valdinoci were visiting the
Justus-Liebig-Universit\"at Giessen, which
they wish to thank for the hospitality.

This work has been supported by the Alexander von Humboldt Foundation, the
ERC grant 277749 {\it E.P.S.I.L.O.N.} ``Elliptic
Pde's and Symmetry of Interfaces and Layers for Odd Nonlinearities'',
the PRIN grant 201274FYK7
``Aspetti variazionali e
perturbativi nei problemi differenziali nonlineari''
and
the ERC grant 339958 {\it Com.Pat.} ``Complex Patterns for
Strongly Interacting Dynamical Systems''.


\end{document}